\documentclass{amsart}
\usepackage{mathptmx, bbm, amscd, amssymb, enumerate, colonequals, mathdots, xcolor}
\usepackage{color}
\definecolor{chianti}{rgb}{0.6,0,0}
\definecolor{meretale}{rgb}{0,0,.6}
\definecolor{leaf}{rgb}{0,.35,0}
\usepackage[colorlinks=true, pagebackref, hyperindex, citecolor=meretale, urlcolor=leaf, linkcolor=chianti]{hyperref}

\newtheorem{theorem}{Theorem}[section]
\newtheorem{lemma}[theorem]{Lemma}
\newtheorem{corollary}[theorem]{Corollary}
\newtheorem{proposition}[theorem]{Proposition}

\theoremstyle{definition}
\newtheorem{definition}[theorem]{Definition}
\newtheorem{example}[theorem]{Example}
\newtheorem{remark}[theorem]{Remark}

\numberwithin{equation}{theorem}

\def\height{\operatorname{ht}}

\def\cd{\operatorname{cd}}
\def\pf{\operatorname{pf}}
\def\Pf{\operatorname{Pf}}
\def\rad{\operatorname{rad}\,}
\def\rank{\operatorname{rank}}
\def\sgn{\operatorname{sgn}}
\def\tr{{\operatorname{tr}}}

\def\GL{\operatorname{GL}}
\def\SL{\operatorname{SL}}
\def\Sp{\operatorname{Sp}}
\def\O{\operatorname{O}}
\def\SO{\operatorname{SO}}

\def\Hom{\operatorname{Hom}}

\def\fraka{\mathfrak{a}}
\def\frakb{\mathfrak{b}}
\def\frakm{\mathfrak{m}}

\def\frakp{\mathfrak{p}}

\def\frakA{\mathfrak{A}}
\def\frakP{\mathfrak{P}}
\def\frakQ{\mathfrak{Q}}
\def\frakS{\mathfrak{S}}

\def\AA{\mathbb{A}}
\def\CC{\mathbb{C}}

\def\NN{\mathbb{N}}
\def\PP{\mathbb{P}}
\def\QQ{\mathbb{Q}}

\def\VV{\mathbb{V}}
\def\ZZ{\mathbb{Z}}

\def\0{\mathbbm{0}}
\def\1{\mathbbm{1}}

\def\ge{\geqslant}
\def\le{\leqslant}
\def\phi{\varphi}
\def\bar{\overline}
\def\tilde{\widetilde}

\def\to{\longrightarrow}
\def\mapsto{\longmapsto}
\def\comp{\mathrm{c}}

\renewcommand{\mod}{\,\operatorname{mod}\,}

\begin{document}
\title[Natural embeddings of classical invariant rings]{When are the natural embeddings of\\
classical invariant rings pure?}

\author{Melvin Hochster}
\address{Department of Mathematics, East Hall, 530 Church St., Ann Arbor, MI 48109, USA}
\email{hochster@umich.edu}

\author{Jack Jeffries}
\address{Department of Mathematics, University of Nebraska-Lincoln, 203 Avery Hall, Lincoln, NE~68588, USA}
\email{jack.jeffries@unl.edu}

\author{Vaibhav Pandey}
\address{Department of Mathematics, Purdue University, 150 N University St., West Lafayette, IN~47907, USA}
\email{pandey94@purdue.edu}

\author{Anurag K. Singh}
\address{Department of Mathematics, University of Utah, 155 South 1400 East, Salt Lake City, UT~84112, USA}
\email{singh@math.utah.edu}

\thanks{M.H. was supported by NSF grants DMS~1902116 and DMS~2200501, J.J. by NSF CAREER Award DMS~2044833, and V.P. and A.K.S. by NSF grants DMS~1801285 and DMS~2101671.}


\begin{abstract}
Consider a reductive linear algebraic group $G$ acting linearly on a polynomial ring $S$ over an infinite field; key examples are the general linear group, the symplectic group, the orthogonal group, and the special linear group, with the classical representations as in Weyl's book: for the general linear group, consider a direct sum of copies of the standard representation and copies of the dual; in the other cases take copies of the standard representation. The invariant rings in the respective cases are determinantal rings, rings defined by Pfaffians of alternating matrices, symmetric determinantal rings, and the Pl\"ucker coordinate rings of Grassmannians; these are the classical invariant rings of the title, with $S^G\subseteq S$ being the natural embedding.

Over a field of characteristic zero, a reductive group is linearly reductive, and it follows that the invariant ring $S^G$ is a pure subring of $S$, equivalently, $S^G$ is a direct summand of~$S$ as an $S^G$-module. Over fields of positive characteristic, reductive groups are typically no longer linearly reductive. We determine, in the positive characteristic case, precisely when the inclusion $S^G\subseteq S$ is pure. It turns out that if $S^G\subseteq S$ is pure, then either the invariant ring $S^G$ is regular, or the group $G$ is linearly reductive.
\end{abstract}
\maketitle

\section{Introduction}

The classical invariant rings that we study here are determinantal rings, rings defined by Pfaffians of alternating matrices, symmetric determinantal rings, and the Pl\"ucker coordinate rings of Grassmannians. Over a field of characteristic zero, these are all invariant rings for classical groups as in Weyl~\cite{Weyl}; by~\cite{Igusa, DeConcini-Procesi, Hashimoto}, these are also invariant rings for the corresponding classical groups over an infinite field of positive characteristic. The embedding $S^G\subseteq S$, for $S$ a polynomial ring and $G$ a classical group, is the \emph{natural embedding} of the title; we describe these in turn. In each case, $K$ is a field of arbitrary characteristic. 

\subsection*{(a)}

Let $Y$ and $Z$ be $m\times t$ and $t\times n$ matrices of indeterminates respectively. Set $S$ to be the polynomial ring $K[Y,Z]$, and take $R$ to be the $K$-subalgebra generated by the entries of the product matrix $YZ$. Then $R$ is isomorphic to the determinantal ring $K[X]/I_{t+1}(X)$, where~$X$ is an $m\times n$ matrix of indeterminates, and $I_{t+1}(X)$ is the ideal generated by its size~$t+1$ minors. The general linear group $\GL_t(K)$ acts~$K$-linearly on $S$ via
\[
M\colon\begin{cases} Y & \mapsto YM^{-1}\\ Z & \mapsto MZ\end{cases}
\]
where $M\in\GL_t(K)$. When the field $K$ is infinite, $R$ is precisely the ring of invariants, see~\cite[\S3]{DeConcini-Procesi} or~\cite[Theorem~4.1]{Hashimoto}.

\subsection*{(b)}

Let $Y$ be a $2t\times n$ matrix of indeterminates and set~$S\colonequals K[Y]$. Let
\begin{equation}
\label{equation:omega}
\Omega\colonequals\begin{pmatrix} 0 & \1 \\ -\1 & 0 \end{pmatrix},
\end{equation}
be the size $2t$ standard symplectic block matrix, where $\1$ is the size $t$ identity matrix. The~$K$-algebra $R\colonequals K[Y^\tr \Omega Y]$ is isomorphic to $K[X]/\Pf_{2t+2}(X)$, where~$X$ is an $n\times n$ alternating matrix of indeterminates, and $\Pf_{2t+2}(X)$ the ideal generated by its principal size~$2t+2$ Pfaffians, see \S\ref{sec:alt}. The \emph{symplectic group}
\[
\Sp_{2t}(K)\colonequals\{M\in\GL_{2t}(K)\ |\ M^\tr\Omega M=\Omega\}
\]
acts $K$-linearly on $S$, where
\[
M\colon Y\mapsto MY\qquad\text{ for }\ M\in\Sp_{2t}(K).
\]
It is readily seen that $Y^\tr\Omega Y\mapsto Y^\tr M^\tr\Omega MY=Y^\tr\Omega Y$ for $M\in\Sp_{2t}(K)$, so the entries of the matrix $Y^\tr\Omega Y$ are invariant under the action; when $K$ is infinite, the invariant ring is precisely the ring $R$, see~\cite[\S6]{DeConcini-Procesi} or~\cite[Theorem~5.1]{Hashimoto}.

\subsection*{(c)}

Let $Y$ be a $d\times n$ matrix of indeterminates. Set~$S\colonequals K[Y]$ and let $R$ be the~$K$-subalgebra generated by the entries of $Y^\tr Y$. Then $R$ is isomorphic to $K[X]/I_{d+1}(X)$, for~$X$ an~$n\times n$ symmetric matrix of indeterminates. The \emph{orthogonal group}
\[
\O_d(K)\colonequals\{M\in\GL_d(K)\ |\ M^\tr M=\1\}
\]
acts~$K$-linearly on $S$ via
\[
M\colon Y\mapsto MY\qquad\text{ for }\ M\in\O_d(K).
\]
Note that $Y^\tr Y\mapsto Y^\tr M^\tr MY=Y^\tr Y$ for $M\in\O_d(K)$, so the entries of $Y^\tr Y$ are invariant under the action; when the field $K$ is infinite of characteristic other than two, the invariant ring is precisely the subring~$R$, see~\cite[\S5]{DeConcini-Procesi}; when $K$ is infinite of characteristic two, as proved in \cite[\S5]{Richman}, the invariant ring has the additional generators
\[
y_{1j}+\dots+y_{dj}\qquad\text{ where }1\le j\le n.
\]

\subsection*{(d)}
Let $Y$ be a $d\times n$ matrix of indeterminates over~$K$, where $d\le n$, and set~$S\colonequals K[Y]$. Let
\[
R\colonequals K[\{\Delta\}]
\]
where $\{\Delta\}$ is the set of size $d$ minors of $Y$. Then~$R$ is the Pl\"ucker coordinate ring of the Grassmannian of~$d$-dimensional subspaces of an~$n$-dimensional vector space. The special linear group~$\SL_d(K)$ acts $K$-linearly on $S$ where
\[
M\colon Y\mapsto MY\qquad\text{ for }\ M\in\SL_d(K).
\]
Each size $d$ minor of $Y$ is fixed under the group action; when $K$ is an infinite field, the invariant ring is precisely~$R$, see~\cite{Igusa} or \cite[\S3]{DeConcini-Procesi}.

\medskip

If $K$ has characteristic zero, the groups $\GL_t(K)$, $\Sp_{2t}(K)$, $\O_d(K)$, and $\SL_d(K)$ are linearly reductive; it follows that, in each case, the invariant ring $R$ is a direct summand of~$S$ as an $R$-module, equivalently that $R\subseteq S$ is pure, see~\S\ref{sec:pure:split:solid} for the equivalence. This then implies a wealth of strong properties for $R$, see~\cite{Boutot, HH:JAMS, Hochster-Roberts, Kempf:MMJ}. Over fields of positive characteristic, these invariant rings maintain favorable properties such as the Cohen-Macaulay property and~$F$-regularity, see \cite[Theorem~7.14]{HH:JAG}, though the groups are typically not linearly reductive; indeed, in positive characteristic, each of the classical groups above admits a representation for which the invariant ring is not Cohen-Macaulay \cite{Kohls}. It is natural to ask if the embeddings (a)--(d) are pure when~$K$ has positive characteristic. We prove:

\begin{theorem}
\label{theorem:main}
Let $K$ be a field of characteristic $p>0$. Fix positive integers~$d,m,n$, and~$t$, and let $R\subseteq S$ denote one of the following inclusions:
\begin{enumerate}[\quad \rm(a)]
\item\label{theorem:main:gl} $K[YZ] \subseteq K[Y,Z]$, where $Y$ and $Z$ are $m\times t$ and $t\times n$ matrices of indeterminates;

\item\label{theorem:main:sp} $K[Y^\tr \Omega Y] \subseteq K[Y]$, where $Y$ is a $2t \times n$ matrix of indeterminates;

\item\label{theorem:main:o} $K[Y^\tr Y] \subseteq K[Y]$, where $Y$ is a $d \times n$ matrix of indeterminates;

\item\label{theorem:main:sl} $K[\{ \Delta \} ] \subseteq K[Y]$, where $Y$ is a $d \times n$ matrix of indeterminates with $d\le n$.
\end{enumerate}
Then $R\subseteq S$ is pure if and only if, in the respective cases,
\begin{enumerate}[\quad \rm(a)]
\item $t=1$ or $\min\{m,n\} \le t$;
\item $n\le t+1$;
\item $d=1$; $d=2$ and $p$ is odd; $p=2$ and $n \le (d+1)/2$; or $p$ is odd and $n\le (d+2)/2$;
\item $d=1$ or $d=n$.
\end{enumerate}
\end{theorem}

Suppose the field $K$ in Theorem~\ref{theorem:main} is infinite; in case~\eqref{theorem:main:o} assume also that the characteristic of $K$ is odd. In this setting, the ring $R$ is the invariant ring $S^G$ for an action of a classical group~$G$ on $S$, as recorded earlier. It is notable that whenever $S^G\subseteq S$ is pure, either the invariant ring $S^G$ is regular, or the group $G$ is linearly reductive:

In~\eqref{theorem:main:gl}, $S^G$ is regular if $\min\{m,n\} \le t$, while if $t=1$, then $\GL_1(K)$ is the torus $K^\times$, which is linearly reductive. For~\eqref{theorem:main:sp}, $S^G$ is regular if $n\le 2t+1$, though $S^G\subseteq S$ is pure in the more restrictive range $n\le t+1$. In case~\eqref{theorem:main:o}, the orthogonal group $\O_d(K)$ is linearly reductive if~$d=1$, and also if $d=2$ and $p$ is odd, as discussed in the proof of Theorem~\ref{theorem:symmetric:purity}; the ring~$S^G$ is regular if $n\le d$, though $S^G\subseteq S$ is pure in a smaller range, and one that depends on the characteristic. Lastly, in~\eqref{theorem:main:sl}, $S^G$ is regular precisely if $d$ equals $1$, $n-1$, or $n$.

The cases~\eqref{theorem:main:gl}--\eqref{theorem:main:sl} of Theorem~\ref{theorem:main} are proven as Theorems~\ref{theorem:determinantal},~\ref{theorem:alternating:purity},~\ref{theorem:symmetric:purity},~and~\ref{theorem:grassmannian:purity}, respectively. In each case, this involves investigating the \emph{nullcone} of the action of $G$ on $S$, namely the ring $S/\frakm_{S^G}S$, where $\frakm_{S^G}$ is the homogeneous maximal ideal of the invariant ring $S^G$ (or, more generally, the ring $S/\frakm_RS$). The study of nullcones goes back at least to Hilbert's proof of the finite generation of invariant rings~\cite{Hilbert}; more recent work includes \cite{Hesselink, Kraft-Schwarz, Kraft-Wallach, Schwarz}. Specifically, Kraft and Schwartz determine, for classical invariant rings of characteristic zero, precisely when the nullcone is reduced or a domain~\cite[Theorem~9.1]{Kraft-Schwarz}. Our paper includes the corresponding results in the positive characteristic case.

The easiest to settle is the $\SL_n(K)$ case: the invariant ring is the homogeneous coordinate ring for the Pl\"ucker embedding of a Grassmannian variety, and the nullcone is a determinantal ring, hence Cohen-Macaulay by Hochster-Eagon~\cite{Hochster-Eagon}; more work is needed in the other cases. For the $\GL_n(K)$ action, the invariant rings are generic determinantal rings, but the nullcone typically fails to be Cohen-Macaulay or even equidimensional; we use the theory of varieties of complexes as introduced by Buchsbaum-Eisenbud~\cite{Buchsbaum-Eisenbud:1975}, and expanded by Kempf~\cite{Kempf:BAMS}, De Concini-Strickland~\cite{DeConcini-Strickland}, and Huneke~\cite{Huneke:TAMS}. We settle the purity question by examining the irreducible components and their intersections.

In the symplectic group~$\Sp_{2n}(K)$ case, the invariant rings are defined by the principal Pfaffians of fixed size of an alternating matrix of indeterminates. It is worth mention that there is much amongst our results that is new even in the case of characteristic zero: for example, for the $\Sp_{2n}(\CC)$ case, Kraft and Schwarz~\cite[Theorem~9.1.3]{Kraft-Schwarz} prove that the nullcone is irreducible and normal; we prove that it is, in addition, Cohen-Macaulay:

\begin{theorem}
\label{theorem:pfaffian:nullcone:intro}
Let $Y$ be a~$2t\times n$ matrix of indeterminates over a field $K$, where $t$ and $n$ are positive integers. Set~$S\colonequals K[Y]$ and take $\frakP$ to be the ideal generated by the entries of the matrix $Y^\tr\Omega Y$, where $\Omega$ is the size $2t$ standard symplectic matrix as displayed in~\eqref{equation:omega}. 

Then $\frakP$ is a prime ideal, and the ring $S/\frakP$ is Cohen-Macaulay.
\end{theorem}

The situation is more complicated in the case of the orthogonal group $\O_d(K)$; the characteristic zero case of parts~\eqref{theorem:symmetric:nullcone:intro:a} and~\eqref{theorem:symmetric:nullcone:intro:b} of the following is~\cite[Theorem~9.1.4]{Kraft-Schwarz}:

\begin{theorem}
\label{theorem:symmetric:nullcone:intro}
Let $Y$ be a~$d\times n$ matrix of indeterminates over a field $K$, where $d$ and $n$ are positive integers. Set~$S\colonequals K[Y]$ and take $\frakA$ to be the ideal generated by the entries of $Y^\tr Y$.
\begin{enumerate}[\quad \rm(1)]
\item Suppose $K$ has characteristic other than $2$. Then:
\begin{enumerate}[\hspace{-.1in}\rm(a)]
\item\label{theorem:symmetric:nullcone:intro:a} The ideal $\frakA$ is radical if and only if $2n\le d$.
\item\label{theorem:symmetric:nullcone:intro:b} If~$K$ contains a primitive fourth root of unity, then $\frakA$ is prime if and only if $2n<d$.
\item If $d$ is odd, or if $2n<d$, then $S/\rad\frakA$ is a Cohen-Macaulay integral domain.
\item Suppose $d$ is even, $2n\ge d$, and $K$ contains a primitive fourth root of unity. Then~$\frakA$ has minimal primes $\frakP$ and $\frakQ$, see Definition~\ref{definition:p:q}, and the rings~$K[Y]/\frakP$ and~$K[Y]/\frakQ$ are Cohen-Macaulay.
\end{enumerate}
\item Suppose $K$ has characteristic two. Then $\frakA$ is not radical; however, $S/\rad\frakA$ is a Cohen-Macaulay integral domain.
\end{enumerate}
\end{theorem}

Theorem~\ref{theorem:pfaffian:nullcone:intro} is part of Theorem~\ref{theorem:pfaffian:nullcone:main}, while Theorem~\ref{theorem:symmetric:nullcone:intro} is covered by Theorems~\ref{theorem:symmetric:nullcone:char2}, \ref{theorem:symmetric:odd:nullcone}, and~\ref{theorem:symmetric:even:nullcone}. It is worth emphasizing that, in all cases
~\eqref{theorem:main:gl}--\eqref{theorem:main:sl} of Theorem~\ref{theorem:main}, the minimal primes of $\frakm_{S^G}S$---the defining ideal of the nullcone---are \emph{perfect ideals}, i.e., define Cohen-Macaulay rings; this supports the maxim \lq\lq Perfection is often hunted for and usually found in generic situations,\rq\rq\ Bruns~\cite{Bruns:compositio}. A key technique used to establish the perfection is that of principal radical systems, introduced by Hochster-Eagon in their study of determinantal rings~\cite{Hochster-Eagon}; this in reviewed in~\S\ref{sec:prs}.

Theorem~\ref{theorem:symmetric:nullcone:intro} is related to work on Lov\'{a}sz-Saks-Schrijver ideals. Given a simple graph~$G$ on a vertex set $\{1,\dots,n\}$, an integer $d$, and a field $K$, let $Y$ be an $n\times d$ matrix of indeterminates over $K$. The Lov\'asz-Saks-Schrijver ideal $L^K_G(d)$ is the ideal of $K[Y]$ generated by the entries of $YY^\tr$ in the positions $(i,j)$ that are edges of~$G$. In \cite{HMSW} and \cite{CW}, the conditions that the ideal $L^K_G(d)$ is radical, prime, or a complete intersection are related to various conditions on $G$ and $d$. Notably, the restriction to simple graphs ensures that the ideals $L^K_G(d)$ are generated by elements whose initial terms are squarefree, allowing for Gr\"obner degeneration techniques; it is easy to see that the ideal $\frakA$ from Theorem~\ref{theorem:symmetric:nullcone:intro} has no squarefree initial ideal.

Let $V$ be a commutative ring, and let $R$ denote either a Pfaffian ring $V[X]/\Pf_{2t+2}(X)$, or a determinantal or symmetric determinantal ring $V[X]/I_{t+1}(X)$. While Theorem~\ref{theorem:main} addresses the purity of the natural embedding $R\subseteq S$ when $V$ is a field of positive characteristic, it remains unresolved whether $R$ is a pure subring of \emph{some} polynomial ring over~$V$. However, when $V$ is the ring of integers or the ring of $p$-adic integers, the following theorem addresses embeddings in arbitrary polynomial rings over~$V$:

\begin{theorem}[{\cite[Theorem~9.1]{JS}}]
Let $V$ denote either the ring of integers $\ZZ$, or a ring of~$p$-adic integers $\widehat{\ZZ_{(p)}}$. Let $d,m,n$, and~$t$ be positive integers.
\begin{enumerate}[\quad \rm(a)]
\item Let $R\colonequals V[X]/I_{t+1}(X)$, where $X$ is an $m\times n$ matrix of indeterminates. Then $R$ is a pure subring of a polynomial ring over $V$ if and only $t=1$ or $\min\{m,n\}\le t$.

\item Let $R\colonequals V[X]/\Pf_{2t+2}(X)$, where $X$ is an $n\times n$ alternating matrix of indeterminates. Then $R$ is a pure subring of a polynomial ring over $V$ if and only if $n\le 2t+1$, i.e., if and only if $R$ is itself a polynomial ring over $V$.

\item Let $R\colonequals V[X]/I_{d+1}(X)$, where $X$ is a symmetric $n\times n$ matrix of indeterminates. Then~$R$ is a pure subring of a polynomial ring over $V$ if and only $n\le d$, or $d=1$, or~$d=2$ and $V=\widehat{\ZZ_{(p)}}$ for $p$ an odd prime.
\end{enumerate}
\end{theorem}

The formulation of the theorem in \cite{JS} is in terms of direct summands rather than pure subrings, but the notions are equivalent when $V$ above is a ring of~$p$-adic integers, from which the remaining assertions follow. Specifically, conditions~\eqref{theorem:solid:1} and~\eqref{theorem:solid:2} in Theorem~\ref{theorem:solid} remain equivalent when $R_0=S_0$ is, more generally, a complete local ring. The proof in this case uses~\cite[Theorem~3.6.17]{Bruns-Herzog}.

\subsection*{Notation:}

For commutative rings $R\subseteq S$, and $M$ a matrix with entries from $S$, we use~$R[M]$ to denote the $R$-algebra generated by the entries of $M$, and $(M)$ or $(M)S$ to denote the ideal of $S$ generated by the entries of $M$. For a product matrix $MN$ one has~$(MN)\subseteq(M)$, so if $N$ is invertible then $(MN)=(M)$.

We use $\1$ for the identity matrix, or $\1_n$ if the size needs to be specified. For a matrix~$M$, we use $M|_s$ to denote the submatrix consisting of the first $s$ columns of $M$; this should not be confused with the notation $M_{\alpha|\beta}$---used only in \S\ref{ssec:sym:odd}---for the submatrix with rows indexed by $\alpha$ and columns indexed by $\beta$.

\section{Pure, split, and solid extensions}
\label{sec:pure:split:solid}

A ring homomorphism $R\to S$ is \emph{pure} if $R\otimes_RM \to S\otimes_RM$ is injective for each~$R$-module $M$. It is readily seen that if $R$ is a direct summand of $S$ as an $R$-module, i.e., the inclusion $R\to S$ is split in the category of $R$-modules, then $R\to S$ is pure. 

A related notion is that of a solid algebra: Let $R$ be an integral domain; following~\cite{Hochster:solid1}, an $R$-algebra $S$ is \emph{solid} if $\Hom_R(S,R)$ is nonzero. If $R$ is a direct summand of $S$ as an $R$-module, it follows that $S$ is a solid $R$-algebra. More generally, we have:

\begin{theorem}[{cf.~\cite[Corollary~2.4]{Hochster:solid1}}]
\label{theorem:solid}
Let $R\to S$ be a degree-preserving inclusion of~$\NN$-graded normal rings that are finitely generated over a field $R_0=S_0$. Set $\frakm_R$ to be the homogeneous maximal ideal of~$R$, and set $d\colonequals\dim R$. Let $E_R$ denote the injective hull of~$R/\frakm_R$ in the category of graded $R$-modules. Consider the following statements:
\begin{enumerate}[\quad \rm(1)]
\item\label{theorem:solid:1} The ring $R$ is a direct summand of $S$ as an $R$-module.
\item\label{theorem:solid:2} The map $R\to S$ is pure.
\item\label{theorem:solid:3} The induced map $R\otimes_RE_R\to S\otimes_RE_R$ is injective.
\item\label{theorem:solid:4} The local cohomology module $H^d_{\frakm_R}(S)$ is nonzero.
\item\label{theorem:solid:5} The $R$-algebra $S$ is solid.
\end{enumerate}

Then~\eqref{theorem:solid:1},~\eqref{theorem:solid:2}, and~\eqref{theorem:solid:3} are equivalent, and imply the equivalent conditions~\eqref{theorem:solid:4} and~\eqref{theorem:solid:5}. If~$R$ is a polynomial ring over a field of positive characteristic, then~\eqref{theorem:solid:1}--\eqref{theorem:solid:5} are equivalent.
\end{theorem}

Since it is an issue that will come up often, we take this opportunity to clarify a point regarding~\eqref{theorem:solid:4}: as $S$ is an $R$-module, so is the local cohomology $H^d_{\frakm_R}(S)$; this is the same~$R$-module as considering the $S$-module $H^d_{\frakm_R S}(S)$ and restricting scalars.

\begin{proof}
The implications \eqref{theorem:solid:1} $\implies$ \eqref{theorem:solid:2} $\implies$ \eqref{theorem:solid:3} are clear; for \eqref{theorem:solid:3} $\implies$ \eqref{theorem:solid:1}, applying the graded dual~$\Hom_R(-,E_R)$ yields the surjection
\[
\CD
\Hom_R(S\otimes_RE_R,E_R) @>>> \Hom_R(E_R,E_R)\\
@| @|\\
\Hom_R(S,R) @>>> R,
\endCD
\]
where the bottom map is simply $\phi\mapsto\phi(1)$.

The equivalence of \eqref{theorem:solid:4} and \eqref{theorem:solid:5} is the graded version of~\cite[Corollary~2.4]{Hochster:solid1}; the proof there is readily modified using instead a homogeneous Noether normalization, and duality in the graded setting.

For \eqref{theorem:solid:2} $\implies$ \eqref{theorem:solid:4}, note that the induced map
\begin{equation}
\label{equation:solid}
H^d_{\frakm_R}(R)\ =\ R\otimes_R H^d_{\frakm_R}(R)\ \to\ S\otimes_R H^d_{\frakm_R}(R)\ =\ H^d_{\frakm_R}(S)
\end{equation}
is injective, where the second equality holds by the right exactness of $H^d_{\frakm_R}(-)$.

Lastly, suppose $R$ is the polynomial ring $K[x_1,\dots,x_d]$ where $K$ is a field of positive characteristic $p$, and that \eqref{theorem:solid:4} holds. The local cohomology module $H^d_{\frakm_R}(R)$ agrees with $E_R$ up to a grading shift, so to show that \eqref{theorem:solid:3} holds, it suffices to verify that the map~\eqref{equation:solid} is injective. Computing $H^d_{\frakm_R}(R)$ using a \v Cech complex on $x_1,\dots,x_d$, its socle is spanned by the cohomology class
\[
\eta\colonequals\left[\frac{1}{x_1 \cdots x_d}\right],
\]
so one need only verify that the image of $\eta$ in $H^d_{\frakm_R}(S)$ is nonzero. Indeed, if this image were zero, then applying the Frobenius map iteratively, the elements
\[
\left[\frac{1}{x_1^{p^e} \cdots x_d^{p^e}}\right]\ \in\ H^d_{\frakm_R}(S)
\]
would be zero for each integer $e\ge 1$. But these generate $H^d_{\frakm_R}(S)$ as an $S$-module.
\end{proof}

The equivalence of the conditions in Theorem~\ref{theorem:solid} may fail when $R$ is a polynomial ring over a field of characteristic zero:

\begin{example}
Set $R$ to be the polynomial ring $\QQ[x_1,x_2,x_3]$, and $S$ to be the hypersurface
\[
\QQ[x_1,x_2,x_3,y_1,y_2,y_3]/\big((x_1x_2x_3)^2-\sum_{i=1}^3y_ix_i^3\big).
\]
Consider the grading with $\deg x_i=1$ and $\deg y_i=3$ for each $i$. A difficult computation of Roberts~\cite{Roberts:lc} shows that $H^3_{(x_1,x_2,x_3)}(S)$ is nonzero, i.e., that the inclusion $R\to S$ satisfies condition~\eqref{theorem:solid:4} in Theorem~\ref{theorem:solid}. However, it does not satisfy~\eqref{theorem:solid:1}, since $(x_1x_2x_3)^2$ is an element of the ideal $(x_1^3,\ x_2^3,\ x_3^3)S$ though not of $(x_1^3,\ x_2^3,\ x_3^3)R$.
\end{example}

Even when $R\to S$ is an inclusion of polynomial rings over a field $K$, the purity may be quite subtle, for example it may depend on the characteristic of $K$. Let $Y$ be a $2\times 3$ matrix of indeterminates over a field~$K$, and set~$S\colonequals K[Y]$. Let $R$ be the $K$-algebra generated by the size $2$ minors of~$Y$. Since the minors are algebraically independent over $K$ in this case, the ring $R$ is a polynomial ring. The inclusion $R\to S$ is pure precisely when $K$ has characteristic zero; this is a special case of the result of the next section, a key ingredient being the vanishing theorem of Peskine-Szpiro, recorded below in the graded setting:

\begin{theorem}[{\cite[Proposition~III.4.1]{PS}}]
\label{theorem:peskine:szpiro}
Let $S$ be a polynomial ring over a field of positive characteristic. If $\fraka$ is a homogeneous ideal such that $S/\fraka$ is Cohen-Macaulay, then
\[
H^k_\fraka(S)=0 \qquad\text{for each }\ k\neq\height\fraka.
\]
\end{theorem}

\section{Pl\"ucker embeddings of Grassmannians}
\label{sec:grassmannian}

The first case of Theorem~\ref{theorem:main} that we address is~\eqref{theorem:main:sl}, namely the case of the special linear group; this ends up being the easiest by far, the nullcones here being the well-studied determinantal rings.

Fix integers $1\le d\le n$. Let $Y$ be a $d\times n$ matrix of indeterminates over a field~$K$, and set~$S\colonequals K[Y]$. Let $R$ denote the $K$-algebra generated by the size $d$ minors of~$Y$. Then~$R$ is the homogeneous coordinate ring, under the Pl\"ucker embedding, of the Grassmannian~$G(d,n)$ of~$d$-dimensional subspaces of an $n$-dimensional vector space. The ring~$R$ is regular when~$d$ equals~$1$, $n-1$, or $n$; in other cases, the relations between the size $d$ minors are quadratic---these are the Pl\"ucker relations,~\cite[Chapter~VII,~\S6]{HP}. The ring $R$ is a Gorenstein unique factorization domain,~\cite{Hochster:schubert, Laksov, Musili}, of dimension~$d(n-d)+1$.

Consider the $K$-linear action of the special linear group~$\SL_d(K)$ on $S$, where
\[
M\colon Y\mapsto MY\qquad\text{ for }\ M\in\SL_d(K).
\]
It is readily seen that the size $d$ minors of~$Y$ are fixed by the group action; when the field~$K$ is infinite, the invariant ring is precisely the subring~$R$, see~\cite{Igusa} or \cite[\S3]{DeConcini-Procesi}. If $K$ is a field of characteristic zero, then the group $\SL_d(K)$ is linearly reductive, and it follows that the invariant ring $R$ is a direct summand of $S$ as an $R$-module. In particular, the inclusion~$R\subseteq S$ is pure when $K$ has characteristic zero. In the case of positive characteristic, we have:

\begin{theorem}
\label{theorem:grassmannian:purity}
Let $K$ be a field of positive characteristic. Let $Y$ be a $d\times n$ matrix of indeterminates where $1\le d\le n$, and set~$S\colonequals K[Y]$. Let $R$ be the $K$-algebra generated by the size $d$ minors of~$Y$. Then the inclusion $R\subseteq S$ is pure if and only if $d=1$ or $d=n$.
\end{theorem}

\begin{proof}
Set $\frakm_R$ to be the homogeneous maximal ideal of $R$. Since the ring $R$ has dimension~$d(n-d)+1$, if the inclusion $R\subseteq S$ is pure, then $H^{d(n-d)+1}_{\frakm_R}(S)$ must be nonzero by Theorem~\ref{theorem:solid}. But~$\frakm_{R}S$ equals the determinantal ideal $I_d(Y)$, which has height $n-d+1$, and defines a Cohen-Macaulay ring $K[Y]/I_d(Y)$, see~\cite{Eagon-Northcott} or \cite{Hochster-Eagon}. But then Theorem~\ref{theorem:peskine:szpiro} implies that
\[
d(n-d)+1\ =\ n-d+1,
\]
i.e., $d=1$ or $d=n$.

Conversely, if $d=1$ or $d=n$, then $R$ is a polynomial ring and $\height(\frakm_{R}S)=\dim R$, so the module $H^{d(n-d)+1}_{\frakm_R}(S)$ is nonzero; hence the inclusion $R\subseteq S$ is pure by Theorem~\ref{theorem:solid}.
\end{proof}

Note that when $d=n-1$ in Theorem~\ref{theorem:grassmannian:purity}, the ring $R$ is regular but $R\subseteq S$ is not pure. The argument above serves as the template for the other cases of Theorem~\ref{theorem:main}, namely we proceed by studying the expansion of the homogeneous maximal ideal $\frakm_R$ of the subring $R$ to the ambient polynomial ring $S$, and analyze the local cohomology obstruction~$H^{\dim R}_{\frakm_R}(S)$. In the remaining cases, the ideal $\frakm_RS$ may be more subtle: in the case of determinantal rings treated next, the ideal $\frakm_RS$ is typically not equidimensional.

\section{Generic determinantal rings}
\label{sec:determinantal}

Let $K$ be a field, and let $Y$ and $Z$ be $m\times t$ and $t\times n$ matrices of indeterminates respectively. Set $S\colonequals K[Y,Z]$, and take $R$ to be the $K$-subalgebra of $S$ generated by the entries of the product matrix $YZ$. Then $R$ is isomorphic to the determinantal ring $K[X]/I_{t+1}(X)$, where~$X$ is an $m\times n$ matrix of indeterminates, and $I_{t+1}(X)$ is the ideal generated by its size~$t+1$ minors. The ring $R$ is Cohen-Macaulay by~\cite{Hochster-Eagon}; it is regular precisely if~$\min\{m,n\} \le t$ since this corresponds to $I_{t+1}(X)=0$. Outside of the regular case, $R$ has dimension $mt+nt-t^2$, class group $\ZZ$ by Bruns~\cite{Bruns:cl}, and is Gorenstein precisely if $m$ equals $n$ by Svanes~\cite{Svanes}. 

The general linear group $\GL_t(K)$ acts~$K$-linearly on $S$ via
\[
M\colon\begin{cases} Y & \mapsto YM^{-1}\\ Z & \mapsto MZ.\end{cases}
\]
where $M\in\GL_t(K)$. When $K$ is infinite, the ring $R$ is precisely the ring of invariants for this action, see~\cite[\S3]{DeConcini-Procesi} or~\cite[Theorem~4.1]{Hashimoto}. If, moreover, the field $K$ has characteristic zero, then $\GL_t(K)$ is linearly reductive, so the ring extension $R\to S$ is pure.

\subsection{Irreducible components of the nullcone}

A complex of $K$-vector spaces
\[
\CD
K^{b_0} @<M_1<< K^{b_1} @<M_2<< \cdots @<M_h<< K^{b_h}
\endCD
\]
can be regarded as a point in affine space using the entries of the matrices $M_k$. Setting $r_k$ to be the rank of $M_k$, the matrices satisfy the rank conditions $r_1\le b_0$, and $r_h\le b_h$, and
\[
r_k + r_{k+1}\le b_k \qquad\text{for }\ 1\le k\le h-1.
\]
Given sequences $(b_0,\dots,b_h)$ and $(r_1,\dots,r_h)$ satisfying these rank conditions, consider matrices of indeterminates $X_k$ of size $b_{k-1}\times b_k$ for $1\le k\le h$. The corresponding \emph{variety of complexes} is the algebraic set defined by the vanishing of the entries of the matrices~$X_kX_{k+1}$ and the determinantal ideals $I_{r_k+1}(X_k)$. When $K$ has characteristic zero, these varieties were shown to be Cohen-Macaulay and normal, with rational singularities, by Kempf~\cite{Kempf:BAMS} using~\cite{Kempf:Invent}. The Cohen-Macaulay property is proved in arbitrary characteristic by Huneke~\cite[Theorem 6.2]{Huneke:TAMS} using principal radical systems, and by De~Concini-Strickland~\cite[Theorem 2.7]{DeConcini-Strickland} using Hodge algebra methods; however, as pointed out by Tchernev~\cite[Example~9.2]{Tchernev}, the Hodge algebra structure of~\cite{DeConcini-Strickland} is not correct, though the assertions can be obtained instead using Gr\"obner bases as in~\cite{Tchernev}. See also~\cite{Musili:Seshadri} and the discussion in the proof of~\cite[Theorem~8.6]{CW}. The normality is~\cite[Theorem~7.1]{Huneke:TAMS}.

Returning to our setting where $Y$ and $Z$ are $m\times t$ and $t\times n$ matrices of indeterminates, and $S=K[Y,Z]$, one has $h=2$ and the complex at hand is
\[
\CD
S^m @<Y<< S^t @<Z<< S^n.
\endCD
\]
The papers above give:

\begin{theorem}[{\cite{DeConcini-Strickland, Huneke:TAMS, Kempf:BAMS, Musili:Seshadri, Tchernev}}]
\label{theorem:determinant:minprimes}
Let $K$ be a field. Fix positive integers $m,n$, and $t$, and set~$S\colonequals K[Y,Z]$ where~$Y$ and $Z$ are, respectively, $m\times t$ and~$t\times n$ matrices of indeterminates. For nonnegative integers $i,j$ with $i+j\le t$, set
\[
\frakp_{i,j}\colonequals I_{i+1}(Y) + I_{j+1}(Z) + (YZ)S,
\]
where $(YZ)S$ is the ideal generated by the entries of the matrix $YZ$. Then:
\begin{enumerate}[\quad \rm(1)]
\item For each $i,j$, the ring $S/\frakp_{i,j}$ is a Cohen-Macaulay normal domain.
\item If $i\le m$ and $j\le n$, then $\height(\frakp_{i,j}) = (m-i)(t-i) + (n-j)(t-j) + ij$.
\item The radical of $(YZ)S$ is the intersection of the prime ideals $\frakp_{i,j}$ with $i+j=t$.
\end{enumerate}
\end{theorem}

It is perhaps amusing to note that varieties of complexes with $h=1$ give us determinantal rings, their Cohen-Macaulay property being used in the $\SL_d(K)$ case of Theorem~\ref{theorem:main}.

\subsection{The purity of the embedding}

We next settle the $\GL_t(K)$ case of Theorem~\ref{theorem:main}:

\begin{theorem}
\label{theorem:determinantal}
Let $K$ be a field of positive characteristic. Fix positive integers $m,n,t$, and consider the inclusion $\phi\colon K[YZ]\to K[Y,Z]$ where $Y$ and $Z$ are, respectively, $m\times t$ and~$t\times n$ matrices of indeterminates. Then $\phi$ is pure if and only if $t=1$, or $m\le t$, or $n\le t$.
\end{theorem}

\begin{proof}
We claim that if the inclusion $\phi\colon K[YZ]\to K[Y,Z]$ is pure for a fixed triple of positive integers $(m,n,t)$, then purity holds as well for the inclusion of the $K$-algebras corresponding to a triple $(m',n',t)$ with $m'\le m$ and $n'\le n$.

To see this, set $Y'$ to be the matrix consisting of the first $m'$ rows of~$Y$, and $Z'$ to be the matrix consisting of the first $n'$ columns of $Z$, and consider the $\NN$-grading on~$K[Y,Z]$ where the indeterminates from the submatrices $Y'$ and $Z'$ have degree $0$, as does $K$, while the remaining indeterminates have degree~$1$, so that ${K[Y,Z]}_0\ =\ K[Y',Z']$. Then
\[
{K[YZ]}_0\ =\ K[Y'Z'],
\]
so $K[Y'Z']$ is a pure subring of $K[YZ]$. Since we are assuming $K[YZ]\to K[Y,Z]$ is pure, it follows that the composition
\[
K[Y'Z']\ \subseteq\ K[YZ]\ \subseteq\ K[Y,Z]
\]
is pure as well, but then so is $K[Y'Z']\subseteq K[Y',Z']$. This proves the claim; similar reduction arguments will be used for other matrix families later in the paper.

Set $S\colonequals K[Y,Z]$ and $R\colonequals K[YZ]$. We next prove that $\phi$ is pure in the cases claimed in the theorem. When $t=1$, the ring $R$ coincides with the Segre product of the polynomial rings~$K[Y]$ and $K[Z]$, which is a pure subring of $S$. For the case $m\le t$, in light of the reduction step, it suffices to establish the purity when~$m=t$ and $n\ge t$. In this case the ring~$R$ has dimension $mn$, specifically the matrix entries
\[
x_{ij}\colonequals {(YZ)}_{ij}
\]
are algebraically independent over $K$, and hence form a homogeneous system of parameters for~$R$. By Theorem~\ref{theorem:solid}, it suffices to show that $H^{mn}_{\frakm_R}(S)$ is nonzero; we show that
\[
\left[\frac{1}{\prod x_{ij}}\right]\ \in\ H^{mn}_{\frakm_R}(S)
\]
is a nonzero element, equivalently that for each $k\ge 1$, one has
\[
(\prod x_{ij})^{k-1}\ \notin\ (x_{11}^k,\dots,x_{mn}^k)S.
\]
It is enough to show the above after specializing the entries of $Y$ to the~$t\times t$ identity matrix. This specialization maps $YZ$ to $Z$, with the image of $S$ being the polynomial ring $K[Z]$. The above display then takes the form
\[
(\prod z_{ij})^{k-1}\ \notin\ (z_{11}^k,\dots,z_{mn}^k)K[Z],
\]
which is immediately seen to hold. The case $n\le t$ is much the same.

Next, suppose $t\ge2$. It remains to prove that $\phi\colon K[YZ]\to K[Y,Z]$ is not pure if $m>t$ and $n>t$. By the reduction step at the beginning of the proof, it suffices to show that $\phi$ is not pure in the case $m=t+1=n$. In this case, the ring $R=K[YZ]$ is a hypersurface of dimension $t^2+2t$, so it suffices by Theorem~\ref{theorem:solid} to show that the local cohomology module 
\[
H^{t^2+2t}_{\frakm_R}(S)
\]
is zero, where $\frakm_R$ is the homogeneous maximal ideal of $R$. The minimal primes of the ideal~$\frakm_RS$ are described by Theorem~\ref{theorem:determinant:minprimes}; in the notation of that theorem, these are the primes $\frakp_{0,t},\ \frakp_{1,t-1},\ \dots,\ \frakp_{t,0}$. With $\cd$ denoting the cohomological dimension, we shall prove that for each integer $k$ with $0\le k\le t$, one has
\begin{equation}
\label{equation:det:cd:bound}
\cd\left(\frakp_{0,t}\cap\frakp_{1,t-1}\cap\dots\cap\frakp_{k,t-k}\right)\ \le\ t^2+t+1,
\end{equation}
from which it follows that $\cd(\frakm_RS)\le t^2+t+1$; since $t\ge 2$, one has $t^2+t+1<t^2+2t$.

We first claim that 
\begin{multline}
\label{equation:det:cd:inductive}
\cd\left(\frakp_{0,t}\cap\frakp_{1,t-1}\cap\dots\cap\frakp_{k,t-k}\right)\ \le\ \max\left\{\cd(\frakp_{0,t}),\ \cd(\frakp_{1,t-1}),\ \dots,\ \cd(\frakp_{k,t-k}),\right. \\
\left. \cd(\frakp_{0,t-1})-1,\ \cd(\frakp_{1,t-2})-1,\ \dots,\ \cd(\frakp_{k-1,t-k})-1\right\}.
\end{multline}
Quite generally, for ideals $\fraka$ and $\frakb$ of $S$, the Mayer-Vietoris sequence
\[
\CD
@>>> H^i_{\fraka}(S)\oplus H^i_{\frakb}(S) @>>> H^i_{\fraka\cap\frakb}(S) @>>> H^{i+1}_{\fraka+\frakb}(S) @>>>
\endCD
\]
shows that
\[
\cd(\fraka\cap\frakb)\ \le\ \max\left\{\cd(\fraka),\ \cd(\frakb),\ \cd(\fraka+\frakb)-1 \right\}.
\]
Using this for the ideals $\fraka\colonequals\frakp_{0,t}\cap\frakp_{1,t-1}\cap\dots\cap\frakp_{k,t-k}$ and $\frakb\colonequals\frakp_{k+1,t-k-1}$, one has
\begin{multline*}
\cd\left([\frakp_{0,t}\cap\frakp_{1,t-1}\cap\dots\cap\frakp_{k,t-k}]\cap\frakp_{k+1,t-k-1}\right)\ \le\ \max\left\{
\cd\left(\frakp_{0,t}\cap\frakp_{1,t-1}\cap\dots\cap\frakp_{k,t-k}\right),\right. \\
\left. \cd(\frakp_{k+1,t-k-1}),\
\cd\left([\frakp_{0,t}\cap\frakp_{1,t-1}\cap\dots\cap\frakp_{k,t-k}]+\frakp_{k+1,t-k-1}\right)-1
\right\}.
\end{multline*}
Up to taking radicals, the ideal
\[
[\frakp_{0,t}\cap\frakp_{1,t-1}\cap\dots\cap\frakp_{k,t-k}]+\frakp_{k+1,t-k-1}
\]
coincides with
\begin{multline*}
(\frakp_{0,t}+\frakp_{k+1,t-k-1})\cap(\frakp_{1,t-1}+\frakp_{k+1,t-k-1})\cap\dots\cap(\frakp_{k,t-k}+\frakp_{k+1,t-k-1})\\
=\ \frakp_{0,t-k-1}\cap\frakp_{1,t-k-1}\cap\dots\cap\frakp_{k,t-k-1}\ =\ \frakp_{k,t-k-1},
\end{multline*}
since $\frakp_{i_1,j_1}+\frakp_{i_2,j_2}=\frakp_{i,j}$ for $i\colonequals\min\{i_1,i_2\}$ and $j\colonequals\min\{j_1,j_2\}$. It follows that
\begin{multline*}
\cd\left(\frakp_{0,t}\cap\frakp_{1,t-1}\cap\dots\cap\frakp_{k+1,t-k-1}\right)\ \le\ \max\left\{
\cd\left(\frakp_{0,t}\cap\frakp_{1,t-1}\cap\dots\cap\frakp_{k,t-k}\right),\right. \\
\left. \cd(\frakp_{k+1,t-k-1}),\ \cd(\frakp_{k,t-k-1})-1\right\}.
\end{multline*}
Using this inductively, one obtains~\eqref{equation:det:cd:inductive}.

Since the rings $S/\frakp_{i,j}$ are Cohen-Macaulay for $i+j\le t$, Theorem~\ref{theorem:peskine:szpiro} implies that $\cd(\frakp_{i,j}) = \height(\frakp_{i,j})$. Consequently~\eqref{equation:det:cd:inductive} gives
\begin{multline*}
\cd\left(\frakp_{0,t}\cap\frakp_{1,t-1}\cap\dots\cap\frakp_{k,t-k}\right)\ \le\ \max\left\{\height(\frakp_{0,t}),\ \height(\frakp_{1,t-1}),\ \dots,\ \height(\frakp_{k,t-k}),\right. \\
\left. \height(\frakp_{0,t-1})-1,\ \height(\frakp_{1,t-2})-1,\ \dots,\ \height(\frakp_{k-1,t-k})-1\right\}.
\end{multline*}
Using the formula for $\height(\frakp_{i,j})$ from Theorem~\ref{theorem:determinant:minprimes}, it is readily verified that for each fixed integer $\ell$ with $0\le \ell \le t$, one has
\[
\max\left\{\height(\frakp_{0,\ell}),\ \height(\frakp_{1,\ell-1}),\ \dots,\ \height(\frakp_{\ell,0})\right\}\ =\
\ell^2 -(2t+1)\ell +2t(t+1),
\]
which then yields~\eqref{equation:det:cd:bound}.
\end{proof}

\section{Principal radical systems}
\label{sec:prs}

Our approach to Theorems~\ref{theorem:pfaffian:nullcone:intro} and~\ref{theorem:symmetric:nullcone:intro} is via the technique of principal radical systems, developed by Hochster and Eagon in \cite{Hochster-Eagon}. This is a method used to prove that a given homogeneous ideal in a polynomial ring is prime, and defines a Cohen-Macaulay ring, by constructing a finite family of radical ideals that contains the ideal of interest, and inductively prove primality and the Cohen-Macaulay property for select ideals in the family --- the desired properties are first proved for larger ideals in the family. The power of the technique was first demonstrated in proving that generic determinantal rings are Cohen-Macaulay, a result that we used in the proof of Theorem~\ref{theorem:grassmannian:purity}. It was also used in Huneke's proof~\cite{Huneke:TAMS} of Theorem~\ref{theorem:determinant:minprimes}. Kutz~\cite{Kutz} used principal radical systems to prove that symmetric determinantal rings are Cohen-Macaulay, while the corresponding result for Pfaffians is due to Kleppe-Laksov~\cite{Kleppe-Laksov} and independently Marinov~\cite{Marinov1, Marinov2}.

The technique uses the following lemma from~\cite[Section~5]{Hochster-Eagon}; the proof, being brief, is included for the convenience of the reader.

\begin{lemma}
\label{lemma:prs}
Let $S$ be an $\NN$-graded ring, finitely generated over a field $S_0$. Let $I$ be a homogeneous ideal, and $P$ a homogenous prime ideal such that $I\subseteq P$. Suppose there exists a homogeneous element $x$ of positive degree such that $x\notin P$ and $I+xS$ is a radical ideal.

\begin{enumerate}[\quad \rm(1)]
\item If $xP\subseteq I$, then $I$ is radical.
\item If $\rad I=P$, then $I=P$.
\end{enumerate}
\end{lemma}

\begin{proof}
(1) Let $u$ be a homogeneous element in the radical of $I$. Then, $u = i+xs$ for homogeneous elements $i$ in $I$ and $s$ in $S$. Then, $xs = u-i$ lies in the radical of $I$ and therefore in~$P$. Since $x$ does not belong to $P$, the element $s$ must. But then $xs$ is an element of $xP\subseteq I$, so~$u=i+xs$ belongs to $I$.

(2) Replacing $S$ by $S/I$, it suffices to prove that $S$ is a domain; the prime ideal $P$ is now the nilradical of $S$. Let $u$ be a homogeneous element in $P$. Since $S/xS$ is reduced, $u=xv$ for some $v\in S$. But $xv$ lies in the prime ideal $P$ and $x$ does not, so $v\in P$. Thus, $P=xP$ which, by the graded version of Nakayama's lemma, implies that $P$ is zero.
\end{proof}

We will also need the following elementary lemma for inductively proving the Cohen-Macaulay property along a principal radical system:

\begin{lemma}
\label{lemma:CM}
Let $S$ be an $\NN$-graded ring, finitely generated over a field $S_0$. Let $Q_1$ and~$Q_2$ be ideals such that $S/Q_1$ and $S/Q_2$ are Cohen-Macaulay rings of equal dimension, say $d$, and such that $S/(Q_1+Q_2)$ is Cohen-Macaulay of dimension $d-1$. Then the ring $S/(Q_1\cap Q_2)$ is Cohen-Macaulay of dimension $d$.
\end{lemma}

\begin{proof}
One has an exact sequence of the form
\[
\CD
0 @>>> S/(Q_1\cap Q_2) @>>> S/Q_1 \oplus S/Q_2 @>>> S/(Q_1+Q_2) @>>> 0.
\endCD
\]
The result follows from the local cohomology exact sequence obtained by applying the functor $H^{\bullet}_\frakm(-)$, where $\frakm$ is the homogeneous maximal ideal of $S$.
\end{proof}

The following result will be used in order to employ Lemma~\ref{lemma:prs}.

\begin{lemma}
\label{lemma:minors}
Let $M$ be a matrix with entries from a commutative ring. Fix a positive integer~$c$, and set $M|_c$ to be the submatrix consisting of the first $c$ columns of $M$. Then, for each integer~$b$ with $b>c$, one has
\[
m_{1b}\ I_k(M|_c)\ \subseteq\ I_{k+1}(M)+(m_{11},\ m_{12},\ \dots,\ m_{1c}).
\]
\end{lemma}

\begin{proof}
Working modulo the ideal $I_{k+1}(M)+(m_{11},\ m_{12},\ \dots,\ m_{1c})$, we reuse the notation~$M$ and $m_{ij}$ in the quotient ring, and show that $m_{1b}$ annihilates the ideal $I_k(M|_c)$. If $c$ is less than $k$, then $I_k(M|_c)=0$. Assume $c\ge k$, and fix $b$ and a $k \times k$ minor of $M|_c$. If the minor involves the first row of $M$, it clearly vanishes.

Therefore we may assume that the minor involves $k$ rows other than the first row. Consider the $(k+1) \times (k+1)$ submatrix of $M$ that involves, additionally, the first row and the~$b$-th column of $M$. This matrix has determinant zero, so the result follows.
\end{proof}

\section{Pfaffian rings}
\label{sec:alt}

Let $t$ be a positive integer, and $X$ a $2t \times 2t$ alternating matrix. The \emph{Pfaffian} of $X$ is
\[
\pf X\colonequals\sum_\sigma\sgn(\sigma)x_{\sigma(1)\sigma(2)}x_{\sigma(3)\sigma(4)}\cdots x_{\sigma(2t-1)\sigma(2t)},
\]
where the sum is taken over permutations of $\{1,2,\dots,2t\}$ that satisfy
\[
\sigma(1) < \sigma(3) < \cdots < \sigma(2t-1) \quad\text{ and }\quad \sigma(1)<\sigma(2),\ \dots,\ \sigma(2t-1) < \sigma(2t).
\]
It is readily seen that $(\pf X)^2=\det X$.

For an alternating matrix $X$ with entries from a commutative ring, we use $\Pf_{2t}(X)$ to denote the ideal generated by the Pfaffians of the size $2t$ principal submatrices of $X$.

Suppose $X$ is an $n\times n$ alternating matrix of indeterminates over a field $K$. In this case, the ring~$K[X]/\Pf_{2t+2}(X)$ is a Gorenstein unique factorization domain of dimension
\[
\binom{n}{2} - \binom{n-2t}{2},
\]
with the convention that $\binom{i}{j}=0$ if $i<j$. The ring $K[X]/\Pf_{2t+2}(X)$ is regular precisely if~$n\le 2t+1$, for then $\Pf_{2t+2}(X)=0$. The Cohen-Macaulay property is due to~\cite{Kleppe-Laksov} and~\cite{Marinov1, Marinov2}; the rings are unique factorization domains by \cite{Avramov}, hence Gorenstein.

The ideal $\Pf_4(X)$ is generated by the elements
\[
x_{ij}x_{kl}-x_{ik}x_{jl}+x_{il}x_{jk},\qquad \text{ for }\ 1\le i<j<k<l\le n. 
\]
These are precisely the Pl\"ucker relations for the Grassmannian $G(2,n)$, and~$K[X]/\Pf_4(X)$ is isomorphic to the homogeneous coordinate ring for $G(2,n)$ from \S\ref{sec:grassmannian}.

Let $Y$ be a $2t\times n$ matrix of indeterminates over a field $K$. Set~$S\colonequals K[Y]$, and let $\Omega$ be the size $2t$ standard symplectic block matrix~\eqref{equation:omega}. Then $Y^\tr\Omega Y$ is an alternating matrix of rank $\min\{2t,n\}$. For $X$ an $n \times n$ alternating matrix of indeterminates, the entrywise map
\[
X\to Y^\tr\Omega Y
\]
induces a $K$-algebra isomorphism between $K[X]/\Pf_{2t+2}(X)$ and the subring $R\colonequals K[Y^\tr\Omega Y]$ of $S$. Our goal in this section is to determine when the inclusion~$\phi\colon R\to S$ is pure. The symplectic group~$\Sp_{2t}(K)$ acts $K$-linearly on $S$, where
\[
M\colon Y\mapsto MY\qquad\text{ for }\ M\in\Sp_{2t}(K).
\]
Since $M^\tr\Omega M=\Omega$ for $M\in\Sp_{2t}(K)$, it follows that the entries of $Y^\tr\Omega Y$ are fixed by the group action; when the field $K$ is infinite, the invariant ring is precisely the subring~$R$, see~\cite[\S6]{DeConcini-Procesi} or~\cite[Theorem~5.1]{Hashimoto}. When the field $K$ has characteristic zero, the group~$\Sp_{2t}(K)$ is linearly reductive and it follows that the invariant ring $R$ is a direct summand of $S$ as an $R$-module; hence~$\phi\colon R\to S$ is pure when $K$ has characteristic zero.

\subsection{Symplectic forms and preliminaries}

Let $K$ be a field and $V$ the vector space~$K^{2t}$ with the standard basis. Then $\Omega$ determines the bilinear form $B\colon V\times V\to K$ given by
\begin{equation}
\label{equation:symplectic:form}
(v_1,v_2)\mapsto v_1^\tr\Omega v_2. 
\end{equation}
Note that $B$ is nondegenerate and alternating, i.e., $B(v,v)=0$ for all $v\in V$; in other words,~$B$ is a \emph{symplectic form} on $V$. One has
\[
B(v_1,v_2)\ =\ -B(v_2,v_1)\qquad\text{ for all }v_i\in V.
\]

The matrix for $B$ with respect to the chosen basis is $\Omega$, while a change of basis results in a matrix of the form $C^\tr\Omega C$. In view of this, matrices $M$ and $N$ are \emph{cogredient} if there exists an invertible matrix $C$ such that
\begin{equation}
\label{equation:cogredient}
N\ =\ C^\tr MC.
\end{equation}

A vector subspace $W$ of $V$ is \emph{isotropic} if $B(w_1,w_2)=0$ for all $w_i\in W$, equivalently if~$W\subseteq W^\perp$. Since $B$ is nondegenerate, for any subspace $W$ one has
\[
\rank W + \rank W^\perp\ =\ 2t.
\]
Hence an isotropic subspace of $V$ has rank at most $t$. Any isotropic subspace of $V$ is contained in one that has maximal rank, which is a \emph{Lagrangian subspace}.

\begin{lemma}
\label{lemma:functional:symplectic}
Let $K$ be a field. Consider the vector space $K^{2t}$ equipped with a symplectic form. Let $L$ be a nonzero linear functional on $K^{2t}$, and let
\[
V_1\subseteq V_2\subseteq \dots\subseteq V_m
\]
be isotropic subspaces of $K^{2t}$ with $\rank V_j\le j$ for each $j$, with $m\le t$. Let $k$ be an integer between $1$ and $m$.

Suppose $L$ vanishes on $V_k$. Then there exist isotropic subspaces
\[
W_1\subset W_2\subset\dots\subset W_m
\]
such that, for each $j$, one has $V_j\subseteq W_j$ and $\rank W_j=j$, and $L$ vanishes on $W_k$.
\end{lemma}

\begin{proof} It suffices to consider the case where $m=t$. Denote the symplectic form by $B$, and set~$H\colonequals \ker L$, a codimension one subspace. We construct the subspaces $W_j$ by reverse induction on $j$. If $V_t$ has dimension $t$, simply choose $W_t$ to be $V_t$ itself. If $V_t$ has dimension less than $t$, then $\dim(V_t^\perp)>t$, so $\dim(V_t^\perp \cap H) \ge t$.

If $k<t$, take $W_t$ to be a Lagrangian subspace of $K^{2t}$ containing $V_t$. If $k=t$, since~${V_t\subset H}$, there exists a nonzero vector $x\in (V_t^\perp \cap H) \smallsetminus V$. Then
\[
V_t +Kx\ \subseteq\ (V_t+Kx)^\perp \cap H.
\]
Continuing in this manner, we can extend $V_t$ to a Lagrangian subspace of $K^{2t}$ on which $L$ vanishes.

Assume that the vector spaces $W_{j+1}, W_{j+2}, \dots, W_t$ have been constructed satisfying the required conditions. There are two cases: if $j$ is different from $k$, simply choose $W_j$ of dimension $j$ such that $V_j \subseteq W_j \subseteq W_{j+1}$; this can be done since $V_j$ has dimension at most $j$ and $W_{j+1}$ has dimension $j+1$. If $j$ equals $k$, choose $W_k$ of dimension $k$ such that ${V_k \subseteq W_k \subseteq H \cap W_{k+1}}$; this can indeed be done since $V_k$ has dimension at most $k$, and~$H \cap W_{k+1}$ has dimension at least
\[
(2t-1)+(k+1)-2t\ =\ k.
\]
Finally, since any subspace of an isotropic subspace is isotropic, we are done.
\end{proof}

Let $M$ be a size $2t\times n$ matrix over $K$, satisfying $M^\tr\Omega M = 0$. Then the columns of~$M$ span an isotropic subspace, so $\rank M \le t$, i.e., $I_{t+1}(M)=0$. By the Nullstellensatz, if $Y$ is a size $2t\times n$ matrix of indeterminates over an algebraically closed field $K$, then
\[
I_{t+1}(Y)\ \subseteq\ \rad (Y^\tr\Omega Y),
\]
where $(Y^\tr\Omega Y)$ is the ideal of $K[Y]$ generated by the entries of the matrix $Y^\tr\Omega Y$. We strengthen this next:

\begin{lemma}
\label{lemma:pfaffian:rank}
Let $Y$ be a size $2t\times n$ matrix of indeterminates over a field $K$. Then, in the polynomial ring $K[Y]$, one has
\[
I_{t+1}(Y)\ \subseteq\ (Y^\tr\Omega Y).
\]
\end{lemma}

\begin{proof}
If $n\le t$, there is nothing to prove. If $Y'$ is a truncation of $Y$ obtained by deleting certain columns, then the alternating matrix $Y'^\tr\Omega Y'$ is a truncation of the alternating matrix~$Y^\tr\Omega Y$ obtained by deleting the corresponding columns and rows; thus, it suffices to prove the lemma when $Y$ is size $2t\times (t+1)$.

Next, note that any size $t+1$ minor of $Y$ equals the determinants of a matrix of the form~$Y\# Z$, where $Z$ is a suitable size $2t\times (t-1)$ matrix with entries $0$ and $1$, and $\#$ denotes the concatenation of matrices; for example, for the upper size $t+1$ minor, one may take~$Z$ to be the block matrix $\begin{pmatrix} 0 \\ \1 \end{pmatrix}$.

Thus, it suffices to prove that for all matrices $Z$ of size $2t\times (t-1)$, one has
\[
\det(Y\# Z)\ \in\ (Y^\tr\Omega Y).
\]
Since $\det(Y\# Z)=\pf((Y\# Z)^\tr\Omega(Y\# Z))$, it suffices to prove that
\[
\pf((Y\# Z)^\tr\Omega(Y\# Z))\ \in\ (Y^\tr\Omega Y).
\]
But $(Y\# Z)^\tr\Omega(Y\# Z)$ is a size $2t$ alternating matrix, and $Y^\tr\Omega Y$ its upper-left size $t+1$ submatrix; working modulo the entries of $Y^\tr\Omega Y$, it suffices to check that the Pfaffian of a size~$2t$ alternating matrix of the form
\[
\begin{pmatrix}
0 & A\\
-A^\tr & B
\end{pmatrix}
\]
is zero, where $A$ and $B$ are size $t\times t$, and the first column of $A$ is zero; this is immediate, as the determinant of such a matrix is zero.
\end{proof}

When $t=1$ in Lemma~\ref{lemma:pfaffian:rank}, one has the equality $I_{t+1}(Y)=(Y^\tr\Omega Y)$, as we will see in the following discussion:

\subsection{Secant varieties of Grassmannians}
\label{ssec:secant}

Let $Y$ be a size $2t\times n$ matrix of indeterminates over a field $K$. Set $\frakP$ to be the ideal generated by the entries of $Y^\tr\Omega Y$. While we will prove later that~$\frakP$ is prime and defines a Cohen-Macaulay ring, it is worth mentioning that when $t=1$ one has
\begin{alignat*}2
Y^\tr\Omega Y\ &=\ \begin{pmatrix}
y_{11} & y_{21}\\
\vdots&\vdots\\
y_{1n} & y_{2n}
\end{pmatrix}
\begin{pmatrix}
0 & 1\\
-1 & 0
\end{pmatrix}
\begin{pmatrix}
y_{11} & \cdots & y_{1n}\\
y_{21} & \cdots & y_{2n}
\end{pmatrix}\\
&=\ \begin{pmatrix}
0 & \Delta_{12} & \Delta_{13} & \hdots & \Delta_{1n}\\
-\Delta_{12} & 0 & \Delta_{23} & \hdots & \Delta_{2n}\\
-\Delta_{13} & -\Delta_{23} & 0 & \hdots & \Delta_{3n}\\
\vdots & \vdots & & \ddots & \vdots\\
-\Delta_{1n} & -\Delta_{2n} & -\Delta_{3n} & \hdots & 0
\end{pmatrix},
\end{alignat*}
i.e., $Y^\tr\Omega Y$ is an alternating matrix where, for $i<j$, the matrix entry $(Y^\tr\Omega Y)_{ij}$ is
\[
\Delta_{ij}\colonequals y_{1i}y_{2j}-y_{1j}y_{2i}.
\]
It follows that $\frakP$ coincides with the determinantal ideal $I_2(Y)$ that has height $n-1$, and defines a Cohen-Macaulay ring $K[Y]/\frakP$. The ring $K[Y^\tr\Omega Y]$ is the homogeneous coordinate ring of the Grassmannian $G(2,n)$ under the Pl\"ucker embedding in $\PP^{\binom{n}{2}-1}$.

More generally, for $t\ge 1$, the ring $K[Y^\tr\Omega Y]$ is the homogeneous coordinate ring of the order $t-1$ secant variety $G(2,n)^{t-1}$, i.e., the closure of the union of linear spaces spanned by $t$ points of $G(2,n)$: For $1\le i<j\le n$, the alternating matrix $Y^\tr\Omega Y$ has $ij$-th entry~$B(v_i,v_j)$, where $v_i$ and $v_j$ are the $i$-th and $j$-th columns of $Y$, and $B$ is the symplectic form \eqref{equation:symplectic:form}; specifically,
\[
(Y^\tr\Omega Y)_{ij}\ =\ (y_{1i}y_{t+1,j}-y_{1j}y_{t+1,i})\ +\ \cdots\ +\ (y_{ti}y_{2t,j}-y_{tj}y_{2t,i}).
\]
In particular,
\[
\dim G(2,n)^{t-1}\ =\ \binom{n}{2} - \binom{n-2t}{2}-1.
\]
Recall that for an irreducible closed projective variety $X$ of dimension $d$ in $\PP^N$, the \emph{expected dimension} of the order $s$ secant variety $X^s$ is $\min\{N,\ ds+d+s\}$; when $\dim X^s$ is less than the expected dimension, $X^s$ is \emph{defective}. Using the formula above, it is readily seen that~$G(2,n)^{t-1}$ is defective precisely if $t\ge 2$ and $n\ge 2t+2$, confer \cite[Theorem~2.1]{CGG}.

The dimension and the defining equations of secant varieties of other Grassmannians are largely unknown.

\subsection{The complete intersection property}

The ideal $\frakP$ has $\binom{n}{2}$ minimal generators corresponding to the upper triangular entries of the alternating matrix $Y^\tr\Omega Y$. We next prove that in the case $n\le t+1$, these generators form a regular sequence, i.e., that~$K[Y]/\frakP$ is a complete intersection ring:

\begin{theorem}
\label{theorem:pfaffian:ci}
Let $Y$ be a~$2t\times n$ matrix of indeterminates over a field $K$, where $n\le t+1$. Set~$S\colonequals K[Y]$ and~$\frakP\colonequals (Y^\tr\Omega Y)S$. Then $S/\frakP$ is a complete intersection ring.
\end{theorem}

\begin{proof}
It suffices to prove that $K[Y]/\frakP$ is a complete intersection ring after specializing the entries of the rows indexed
\[
1,\ 2,\ \dots,\ t+1-n,\qquad t+1,\ t+2,\ \dots,\ 2t+1-n
\]
to zero, since this leaves the number of defining equations unchanged. We may hence assume that the matrix $Y$ has $2t-2(t+1-n)=2n-2$ rows, i.e., that $Y$ is size $2(n-1)\times n$, equivalently, size $2t\times (t+1)$. Next, specialize the entries of $Y$ to the corresponding entries of the matrix
\[
\bar{Y}\colonequals\begin{pmatrix}
0 & y_{12} & y_{13} & \cdots & y_{1t} & y_{1,t+1}\\
0 & 0 & y_{23} & \cdots & y_{2t} & y_{2,t+1}\\
0 & 0 & 0 & & y_{3t} & y_{3,t+1}\\
\vdots & \vdots & \vdots & \ddots & & \vdots\\
0 & 0 & 0 & \cdots & 0 & y_{t,t+1}\\
\hline
y_{12} & y_{13} & y_{14} & \cdots & y_{1n} & 0\\
y_{23} & y_{24} & y_{25} & & 0 & 0\\
y_{34} & y_{35} & y_{34} & & 0 & 0\\
\vdots & & \iddots & \iddots & \vdots & \vdots\\
y_{t,t+1} & 0 & 0 & \cdots & 0 & 0
\end{pmatrix}.
\]
This entails killing
\[
\binom{t+1}{2}+t(t+1)\ =\ 2t(t+1) -\binom{t+1}{2}
\]
linear forms in $K[Y]$. Since $K[\bar{Y}]/(\bar{Y}^\tr\Omega \bar{Y})$ is Artinian, it follows that $K[Y]/\frakP$ is a complete intersection ring.
\end{proof}

\begin{corollary}
\label{corollary:pfaffian:ci}
Let $Y$ be a~$2t\times n$ matrix of indeterminates over a field $K$, where $n\le t$. Set~$S\colonequals K[Y]$ and~$\frakP\colonequals (Y^\tr\Omega Y)S$. Let $\fraka$ be an ideal generated by $k$ distinct entries from rows $1$ and $t+1$ of the matrix $Y$. Then
\[
\dim S/(\frakP+\fraka)\ =\ 2nt-\binom{n}{2}-k,
\]
i.e., $S/(\frakP+\fraka)$ is a complete intersection ring.
\end{corollary}

\begin{proof}
As seen in the previous proof, the generators of the ideal $\fraka$ form part of a system of parameters for $S/\frakP$.
\end{proof}

The following lemma will be used to prove the irreducibility of certain algebraic sets of the form $V(\frakP+\fraka)$ in Proposition~\ref{proposition:irreducible:pfaffian}:

\begin{lemma}
\label{lemma:pfaffian:nzd}
Let $Y$ be a $2t\times t$ matrix of indeterminates over a field $K$. Set~$S\colonequals K[Y]$ and
\[
I\colonequals (Y^\tr\Omega Y)S + (y_{12},\dots,y_{1t},\ y_{t+1,1},\dots,y_{t+1,t})S.
\]
Let $\Delta$ be the upper $t \times t$ minor of $Y$. Then $\Delta$ is a nonzerodivisor on $S/I$.
\end{lemma}

\begin{proof}
It suffices to consider the case where the field $K$ is algebraically closed. Since $S/I$ is a complete intersection ring by the corollary above, we need to show that $\Delta$ does not belong to any minimal prime of $I$.

Let $G$ be the subgroup of $\Sp_{2t}(K)$ consisting of matrices $M\colonequals(m_{ij})$ with
\begin{alignat*}3
m_{11} &= 1 &&= m_{t+1,t+1} && \\
m_{1i} &= 0 &&= m_{i1} && \quad \text{for}\ i\neq 1\\
m_{t+1,i} &= 0 &&= m_{i,t+1} && \quad \text{for} \ i\neq t+1.
\end{alignat*}
Deleting rows and columns $1$ and $t+1$ shows that $G$ is isomorphic to $\Sp_{2t-2}(K)$, and is hence a connected algebraic group. The action of $G$ on $S$ via $M\colon Y\mapsto MY$ induces an action on $S/I$, and thus on the (necessarily finite) set of minimal primes of $S/I$. Since $G$ is connected the action must be trivial, i.e.,~$G$ stabilizes each minimal prime of $I$.

Suppose a minimal prime $P$ of $I$ contains $\Delta$. Using the fact that $G$ stabilizes $P$, we shall first show that $P$ contains each maximal minor of $Y$ that involves the first row. Since row~$t+1$ of $Y$ is contained in $I$, hence in $P$, we need only consider maximal minors of $Y$ that involve the first row, and not row $t+1$. We use $M\cdot\Delta$ to denote the image of $\Delta$ under an element $M$ of $G$.

Let $\alpha$ be a size $t$ subset of the row indices $\{1,\dots,2t\}$ such that $1\in\alpha$ and $t+1\notin\alpha$. We use $Y_\alpha$ for the square submatrix with rows $\alpha$, and set $\ell(\alpha)$ to be the number of indices~$a\in \alpha$ such that $a\le t$ and $a+t\in \alpha$. The proof that $\det(Y_\alpha)\in P$ is by induction on $\ell(\alpha)$.

For the case $\ell(\alpha)=0$, proceed by induction on the number $w$ of $a\in \alpha$ with $a>t$. When~$w=0$, one has $\det(Y_\alpha)=\Delta$, which is an element of $P$. For the inductive step, consider the $2t\times 2t$ matrix $M$ with
\[
M_{ij} \colonequals\begin{cases}
1 &\text{if }\ i=j,\ \text{ or if }\ j=i+t \in \alpha,\\
0 &\text{otherwise}.
\end{cases}
\]
Observe that $M\in G$, and that the matrix $MY$ is obtained from $Y$ by the row operations where row $i+t$ is added to row~$i$ whenever $i\le t$ and $i+t\in\alpha$. It follows that $M\cdot\Delta$ is the determinant of the $t\times t$ matrix whose $i$-th row is the sum of rows $i$ and $i+t$ of $Y$ if~$i\le t$ and~$i+t\in\alpha$, and is row $i$ of $Y$ otherwise. By the linearity of determinants along a row,~$M\cdot\Delta$ is the sum of $t\times t$ minors of~$Y$, each of which is indexed by a set of rows $\beta$ with $\ell(\beta)=0$. One of these is~$\det(Y_\alpha)$, while the others have fewer indices greater than $t$. Using $M\cdot\Delta\in P$ and the inductive hypothesis, it follows that $\det(Y_\alpha)\in P$, settling the case $\ell(\alpha)=0$.

Next, fix $\alpha$ with $\ell(\alpha)>0$. Let $i,j\in\{1,\dots,t\}$ be such that $i, i+t\in \alpha$ and $j,j+t\notin \alpha$; such a $j$ exists by cardinality reasons. Let $\alpha' = \alpha \smallsetminus \{i,i+t\}$. Observe that each of
\[
\ell(\alpha'\cup\{i,j\}),\quad \ell(\alpha'\cup\{i,j+t\}),\quad \ell(\alpha'\cup\{i+t,j\}),\quad \ell(\alpha'\cup\{i+t,j+t\})
\]
is strictly less than $\ell(\alpha)$. Let $M$ be the $2t \times 2t$ matrix with 
\[
M_{ab} \colonequals\begin{cases}
1 &\text{if }\ a=b,\ \text{ or }\ (a,b) = (i,j+t),\ \text{ or }\ (a,b) = (j,i+t),\\
0 &\text{otherwise}.
\end{cases}
\]
Note that $M\in G$, and that the matrix $MY$ is obtained from $Y$ by row operations where the~$(j+t)$-th row is added to the $i$-th row, and the $(i+t)$-th row is added to the $j$-th row. Hence, up to choices of signs, $M\cdot\det(Y_{\alpha'\cup\{i,j\}})$ is the sum of 
\[
\det(Y_{\alpha'\cup\{i,j\}}),\quad
\det(Y_{\alpha'\cup\{i,i+t\}}),\quad
\det(Y_{\alpha' \cup \{j,j+t\}}),\quad
\det(Y_{\alpha' \cup \{i+t,j+t\}}).
\]
By the inductive hypothesis, $\det(Y_{\alpha'\cup\{i,j\}})$ and $\det(Y_{\alpha'\cup\{i+t,j+t\}}$ are elements of the prime~$P$, as is $\det(Y_{\alpha'\cup\{i,j\}})$ and hence $M\cdot\det(Y_{\alpha'\cup\{i,j\}})$. It follows that, with a sign choice, one of 
\begin{equation}
\label{equation:pfaffian:minor:1}
\det(Y_{\alpha'\cup\{i,i+t\}})\ \pm\ \det(Y_{\alpha'\cup\{j,j+t\}})
\end{equation}
is an element of $P$. We claim that there exists a Pl\"ucker relation in $K[Y]$ of the form
\begin{multline}
\label{equation:plucker}
\det(Y_{\alpha'\cup\{i,i+t\}}) \det(Y_{\alpha'\cup\{j,j+t\}})
\ \pm\ \det(Y_{\alpha'\cup\{i,j\}}) \det(Y_{\alpha'\cup\{i+t,j+t\}})\\
\pm\ \det(Y_{\alpha'\cup\{i,j+t\}}) \det(Y_{\alpha'\cup\{i+t,j\}})\ =\ 0.
\end{multline}
This may be verified, for example, by passing to a dense open subset of matrices where the rows $\alpha'\cup\{i,i+t\}$ form a basis for $K^t$, and multiplying on the right by an invertible matrix so as reduce to the case where these rows are the standard basis for $K^t$. The equality is now readily checked.

Since the other terms in~\eqref{equation:plucker} belong to $P$ by the induction hypothesis, one obtains
\begin{equation}
\label{equation:pfaffian:minor:2}
\det(Y_{\alpha'\cup\{i,i+t\}})\det(Y_{\alpha'\cup\{j,j+t\}})\ \in\ P.
\end{equation}
Combining \eqref{equation:pfaffian:minor:1} and \eqref{equation:pfaffian:minor:2}, bearing in mind that $P$ is prime, it follows that 
\[
\det(Y_\alpha)\ =\ \det(Y_{\alpha'\cup\{i,i+t\}})\ \in\ P,
\]
completing the proof that $P$ contains each $t\times t$ minor of $Y$ that involves the first row.

If $P$ contains $y_{11}$, Corollary~\ref{corollary:pfaffian:ci} gives a contradiction; it follows that the prime ideal $P$ must contain each size $t-1$ minor of the last $t-1$ columns of $Y$. 

Let $Y'$ be the $2(t-1)\times t$ submatrix obtained by deleting rows $1$ and $t+1$ of $Y$, and $\Omega'$ be the size $2t-2$ standard symplectic block matrix. Set $I'$ to be the ideal of $K[Y']$ generated by the entries of $Y'^\tr\Omega'Y'$ along with the size $t-1$ minors of the last~$t-1$ columns of $Y'$. On an open dense subset of $V(I')$, the last column belongs to the span of colums $2,3,\dots,t-1$. Since the dimension of the Pfaffian nullcone corresponding to a $2(t-1)\times(t-1)$ matrix is 
\[
2(t-1)^2-\binom{t-1}{2}
\]
by Corollary~\ref{corollary:pfaffian:ci}, it follows that
\[
\dim V(I')\ \le\ 2(t-1)^2-\binom{t-1}{2}+(t-2).
\]
Accounting for the matrix entry $y_{11}$, this implies
\[
\dim V(P)\ \le\ 2(t-1)^2-\binom{t-1}{2}+(t-2)+1\ =\ 2t^2-\binom{t}{2}-2t.
\]
But then
\[
\dim V(P)\ <\ \dim V(I)\ =\ 2t^2-\binom{t}{2}-(2t-1),
\]
where the equality uses, again, Corollary~\ref{corollary:pfaffian:ci}. This is not possible since $P$ is a minimal prime of $I$.
\end{proof}

The following proposition serves as a building block in the proof of Theorem~\ref{theorem:pfaffian:nullcone}; the primality of $I_a$ or $I'_a$ does not follow immediately from the proof here, in view of the initial reduction step, though it will be obtained later as part of Theorem~\ref{theorem:pfaffian:nullcone}.

\begin{proposition}
\label{proposition:irreducible:pfaffian}
Let $Y$ be a~$2t\times n$ matrix of indeterminates over an algebraically closed field~$K$, where $n\le t$. Set~$S\colonequals K[Y]$ and~$\frakP\colonequals (Y^\tr\Omega Y)S$. For $a$ with $0\le a\le n-1$, set
\[
I_a \colonequals \frakP + (y_{11},\dots,y_{1a})
\quad\text{ and }\quad
I'_a \colonequals \frakP + (y_{11},\dots,y_{1n}, y_{t+1,1},\dots,y_{t+1,a}).
\]
Then the algebraic sets $V(I_a)$ and $V(I'_a)$ are irreducible.
\end{proposition}

\begin{proof}
Since the projection map onto the first $n$ columns provides a surjection of algebraic sets, it suffices to prove each result in the case $n=t$. Let $\Delta$ be the upper $t\times t$ minor of $Y$.

We first consider $I_a$. In this case, Corollary~\ref{corollary:pfaffian:ci} and Lemma~\ref{lemma:pfaffian:nzd}---after permuting columns---show that $\Delta$ is a nonzerodivisor modulo $I_a$. Write $Y$ as
\[
\begin{pmatrix}Y_1\\ Y_2\end{pmatrix},
\]
where $Y_1$ and $Y_2$ are size $t\times t$. Since $Y_1$ is invertible over the ring~$S_\Delta$, one has $S_\Delta=K[Y_1,\ Z]_\Delta$, where the entries of $Y_1$ and $Z\colonequals Y_2Y_1^{-1}$ are algebraically independent over $K$. Note that
\[
YY_1^{-1}\ =\ \begin{pmatrix}\1\\ Z\end{pmatrix},
\]
so the ideal $(Y^\tr\Omega Y)S_\Delta$ is generated by the entries of 
\[
(YY_1^{-1})^\tr\Omega(YY_1^{-1})\ =\ 
\begin{pmatrix}\1 & Z^\tr\end{pmatrix}
\begin{pmatrix} 0 & \1 \\ -\1 & 0 \end{pmatrix}
\begin{pmatrix}\1\\ Z\end{pmatrix}\ =\
Z-Z^\tr.
\]
It follows that
\[
I_a S_\Delta\ =\ (Z-Z^\tr)S_\Delta+(y_{11},\dots,y_{1a})S_\Delta.
\]
Since $I_a S_\Delta$ is generated by linear forms belonging to the polynomial ring $K[Y_1,\ Z]$, it is a prime ideal of $S_\Delta$; as $\Delta$ is a nonzerodivisor modulo $I_a$, there is a bijection between the minimal primes of $I_a$ and those of $I_a S_\Delta$. It follows that $I_a$ has a unique minimal prime, i.e., that $V(I_a)$ is irreducible.

In the case of $I'_a$, working again with $n=t$, the ring $S/I'_a$ is a polynomial extension of
\[
K[Y']/(Y'^\tr\Omega' Y'),
\]
where $Y'$ is the $(2t-2)\times t$ matrix of indeterminates obtained by deleting rows $1$ and $t+1$ of $Y$, and $S'\colonequals K[Y']$, and $\Omega'$ is the size $2t-2$ standard symplectic block matrix. It suffices to prove that the ring $S'/(Y'^\tr\Omega' Y')$ has a unique minimal prime. Let $\Delta'$ be the upper left size $t-1$ minor of $Y'$; Lemma~\ref{lemma:pfaffian:nzd} implies that $\Delta'$ is a nonzerodivisor on $S'/(Y'^\tr\Omega' Y')$. Writing the matrix $Y'$ as
\[
Y'\ =\ \begin{pmatrix}Y_1 & W_1\\ Y_2 & W_2\end{pmatrix},
\]
where $Y_1$ and $Y_2$ are square matrices of size $t-1$, one has
\[
\begin{pmatrix}Y_1 & W_1\\ Y_2 & W_2\end{pmatrix} \begin{pmatrix}Y_1^{-1} & -Y_1^{-1}W_1\\ 0 & 1\end{pmatrix}\ =\
\begin{pmatrix}\1 & 0\\ Y_2Y_1^{-1} & W_2-Y_2Y_1^{-1}W_1\end{pmatrix}.
\]
The entries of $Y_1$, $W_1$, $Z_1\colonequals Y_2Y_1^{-1}$ and $Z_2\colonequals W_2-Y_2Y_1^{-1}W_1$ are algebraically independent over $K$, and $S'_{\Delta'}$ may be viewed as $K[Y_1,\ W_1,\ Z_1,\ Z_2]_{\Delta'}$. Since
\[
\begin{pmatrix}\1 & Z_1^\tr\\ 0 & Z_2^\tr\end{pmatrix}
\begin{pmatrix} 0 & \1 \\ -\1 & 0\end{pmatrix}
\begin{pmatrix}\1 & 0\\ Z_1 & Z_2\end{pmatrix}\ =\
\begin{pmatrix}Z_1-Z_1^\tr & Z_2\\ -Z_2^\tr & 0\end{pmatrix},
\]
it follows that
\[
S'_{\Delta'}/(Y'^\tr\Omega' Y')_{\Delta'}\ =\ K[Y_1,\ W_1,\ Z_1,\ Z_2]_{\Delta'}/(Z_1-Z_1^\tr,\ Z_2),
\]
and is hence a domain; in particular, it has a unique minimal prime.
\end{proof}

\subsection{Nullcones of Pfaffian rings are Cohen-Macaulay}

We now set up the principal radical system needed to study the nullcones of Pfaffian rings. Let $Y$ be a~$2t\times n$ matrix of indeterminates over a field $K$, and set $\frakP$ to be the ideal generated by the entries of the matrix $Y^\tr\Omega Y$. Let
\[
\sigma\colonequals(s_0,s_1,s_2,\dots,s_m)
\]
be a sequence of integers with $0\le s_k\le n$ for each $k$, and $s_m=n$. Set
\[
I_\sigma\colonequals\frakP + I_1\big(Y|_{s_0}\big) + I_2\big(Y|_{s_1}\big) + I_3\big(Y|_{s_2}\big) + \dots + I_{m+1}\big(Y|_{s_m}\big),
\]
where $I_{k+1}\big(Y|_{s_k}\big)$ denotes the ideal generated by the size $k+1$ minors of the submatrix consisting of the first $s_k$ columns of $Y$. 

In studying $K[Y]/I_\sigma$ there is little loss of generality in assuming $s_0=0$, since one may replace $Y$ by a smaller matrix; in light of Lemma~\ref{lemma:pfaffian:rank}, one may also stipulate~$m\le t$. Note that for positive integers $j$ and $k$, one has
\[
I_{k+1}\big(Y|_{j+1}\big)\ \subseteq\ I_k\big(Y|_j\big),
\]
so one may restrict to $\sigma$ where the entries are strictly increasing. We say $\sigma$ is \emph{standard} if
\[
0=s_0 < s_1 < s_2 < \dots < s_m=n\quad\text{ and }\quad m\le t.
\]
The ideal $\frakP$ indeed equals $I_\sigma$ for a choice of $\sigma$ that is standard: take
\[
\sigma=\begin{cases}
(0,1,2,\dots,n-1,n) & \text{ if }\ n\le t,\\
(0,1,2,\dots,t-1,n) & \text{ if }\ n>t.\\
\end{cases}
\]
For integers $a$ with $0\le a\le n$, set
\[
J_a\colonequals (y_{11},\ y_{12},\ \dots,\ y_{1a})
\quad\text{ and }\quad 
J'_a\colonequals (y_{11},\ y_{12},\ \dots,\ y_{1n},\ \ y_{t+1,1},\ y_{t+1,2},\ \dots,\ y_{t+1,a}).
\]
Note that if $\sigma\colonequals(s_0,s_1,s_2,\dots,s_m)$ is standard, $m=t$, and $s_{m-1}<a<s_m$, then
\[
I_\sigma+J'_a\ =\ I_{\sigma'}+J'_a\qquad \text{ for }\ \sigma'\colonequals(s_0,s_1,s_2,\dots,s_{m-2},a,s_m),
\]
since rows $1$ and $t+1$ of $Y|_a$ are zero modulo $J'_a$, so Lemma~\ref{lemma:pfaffian:rank} gives
\[
I_t(Y|_a)\ \subseteq\ (Y^\tr\Omega Y) + J'_a.
\]

With this notation, we prove:

\begin{theorem}
\label{theorem:pfaffian:nullcone}
Let $Y$ be a~$2t\times n$ matrix of indeterminates over a field $K$, and set $S\colonequals K[Y]$. Let $\sigma\colonequals(s_0,s_1,s_2,\dots,s_m)$ be a sequence of integers with $0\le s_k\le n$ for each $k$, and $s_m=n$. Let $a$ be an integer with $0\le a\le n$. Then:
\begin{enumerate}[\quad \rm(1)]
\item\label{theorem:pfaffian:nullcone:1} If $\sigma$ is standard, then the algebraic sets $V(I_\sigma+J_{s_k})$ and $V(I_\sigma+J'_{s_k})$ are irreducible for each $k$ with $0\le k\le m$.

\item\label{theorem:pfaffian:nullcone:2} The ideals $I_\sigma+J_a$ and $I_\sigma+J'_a$ are radical; if $\sigma$ is standard, then the ideals $I_\sigma+J_{s_k}$ and $I_\sigma+J'_{s_k}$ are prime for each $k$ with $0\le k\le m$.

\item\label{theorem:pfaffian:nullcone:3} Suppose $\sigma$ is standard. If $a=s_k$ for some $k$ with $0\le k\le m$, then $S/(I_\sigma+J_a)$ is a Cohen-Macaulay integral domain of dimension
\[
m(2t+n-m)-k-\sum_{j=1}^{m-1}s_j.
\]
If $a=s_k$ \,for some $k$ with $0\le k\le m-1$, then $S/(I_\sigma+J'_a)$ is a Cohen-Macaulay integral domain of dimension
\[
m(2t+n-m-1)-k-\sum_{j=1}^{m-1}s_j.
\]
\end{enumerate}
\end{theorem}

\begin{proof}
It suffices to prove the assertions when $K$ is algebraically closed; we indeed work under this assumption. We begin by proving~\eqref{theorem:pfaffian:nullcone:1} for the algebraic set $V(I_\sigma+J_{s_k})$. Consider matrices $B$ of size $2t\times m$ for which the columns span an isotropic subspace, and the first~$k$ entries of the first row are zero. Since $m\le t$, Proposition~\ref{proposition:irreducible:pfaffian} implies that the matrices $B$ are the points of an irreducible algebraic set that we denote $\VV_0$.

For $1\le j\le m$, let $C_j$ be a matrix of size $j\times(s_j-s_{j-1})$, and set $A$ to be the matrix
\begin{equation}
\label{equation:pfaffian:irreducible}
(B|_1C_1)\,\#\,(B|_2C_2)\,\#\,\cdots\,\#\,(B|_mC_m),
\end{equation}
where $\#$ denotes the concatenation of matrices. It is readily seen that $A$ is an element of the algebraic set $V(I_\sigma+J_{s_k})$. The matrices $C_1,\dots,C_m$ may be regarded as the points of an affine space $\VV_1$ of dimension
\[
\sum_{j=1}^m j(s_j-s_{j-1}),
\]
so that the construction~\eqref{equation:pfaffian:irreducible} gives a map
\[
\VV_0\times \VV_1\to V(I_\sigma+J_{s_k}).
\]
Since the image of an irreducible algebraic set is irreducible, it suffices to verify that this map is surjective.

Let $A$ be a matrix in the algebraic set $V(I_\sigma+J_{s_k})$. For $1\le j\le m$, let $V_j$ denote the span of the columns of the truncated matrix $A|_{s_j}$. Consider the symplectic form~\eqref{equation:symplectic:form} on $K^{2t}$, and the linear functional $L$ that is projection to the first coordinate. By Lemma~\ref{lemma:functional:symplectic}, there exist isotropic subspaces
\[
W_1 \subset W_2 \subset \dots \subset W_m
\]
such that $V_j\subseteq W_j$ for each $j$, and $W_j$ has rank $j$. Consider a size $2t\times m$ matrix $B$ such that~$B|_j$ spans $W_j$ for each $j$. Then the columns of $A|_{s_j}$ belong to the column span of $B|_j$ for each $j$, so there exist matrices $C_j$ using which $A$ may be obtained as in~\eqref{equation:pfaffian:irreducible}.

The proof that $V(I_\sigma+J'_{s_k})$ is irreducible is similar: we consider instead matrices $B$ of size $2t\times m$, where the columns span an isotropic subspace, and for which the first row is zero, and the first~$k$ entries of row $t+1$ are zero. Proposition~\ref{proposition:irreducible:pfaffian} implies that such matrices~$B$ are the points of an irreducible algebraic set. The linear functional used when applying Lemma~\ref{lemma:functional:symplectic} is now projection to the $t+1$ coordinate.

The proof of~\eqref{theorem:pfaffian:nullcone:2} is via induction, assuming the result for matrices $Y$ of smaller size, as well as for larger ideals in the family, and applying Lemma~\ref{lemma:prs}. Set $I$ to be either $I_\sigma+J_a$ or~$I_\sigma+J'_a$. In the latter case, assume that $a<n$, since otherwise $K[Y]/(I_\sigma+J'_a)$ arises from the smaller matrix obtained by deleting rows $1$ and $t+1$ of $Y$. To apply Lemma~\ref{lemma:prs}, choose
\[
x\colonequals\begin{cases}
y_{1,a+1}& \text{if }\ I=I_\sigma+J_a,\text{ and }a<n,\\
y_{t+1,1}& \text{if }\ I=I_\sigma+J_n,\\
y_{t+1,a+1}& \text{if }\ I=I_\sigma+J'_a.\\
\end{cases}
\]
Specializing $x$ to $1$ and each other entry to $0$, we obtain a matrix in $V(I)\smallsetminus V(I+xS)$, from which it follows that $I+xS$ is a larger ideal in the family, and hence radical by the inductive hypothesis. If $a=s_k$ for some $k$, then $P\colonequals \rad I$ is prime by~\eqref{theorem:pfaffian:nullcone:1}; since $x\notin P$, Lemma~\ref{lemma:prs} implies that $I=P$, and hence $I$ is prime.

In the remaining cases, there exists an integer $k$ with $s_k<a<s_{k+1}$ and the element $x$ is either $y_{1,a+1}$ or $y_{t+1,a+1}$. Set
\[
\sigma'\colonequals(s_0,s_1,\dots,s_{k-1},a,s_{k+1},\dots,s_m),
\]
and take $P$ to be the prime ideal $I_{\sigma'}+J_a$ or~$I_{\sigma'}+J'_a$ in the respective cases; if $k=0$, then~$\sigma'=(a,s_1,\dots,s_m)$ is not standard, but the primality follows nonetheless from the case of a matrix of size $2t\times(n-a)$. The specialization used earlier shows that $x\notin P$. Using Lemma~\ref{lemma:minors}, one has
\[
y_{1,a+1}I_{k+1}(Y|_a)\ \subseteq\ I_{k+2}(Y|_{a+1})+J_a\ \subseteq\ I_{k+2}(Y|_{s_{k+1}})+J_a\ \subseteq\ I_\sigma+J_a
\]
and 
\[
y_{t+1,a+1}I_{k+1}(Y|_a)\ \subseteq\ I_{k+2}(Y|_{a+1})+J'_a\ \subseteq\ I_{k+2}(Y|_{s_{k+1}})+J'_a\ \subseteq\ I_\sigma+J'_a,
\]
so $xP\subseteq I$ in either case. It follows that $I$ is radical by Lemma~\ref{lemma:prs}.

For~\eqref{theorem:pfaffian:nullcone:3}, let $V$ denote the algebraic set $V(I_\sigma+J_a)$ or $V(I_\sigma+J'_a)$. We first compute the dimension of $V$. In each case, $V$ has an open subset $U$ in which each matrix has the property that the submatrix consisting of the columns indexed
\begin{equation}
\label{equation:pfaffian:column:set}
s_0+1,\ s_1+1,\ \dots,\ s_{m-1}+1
\end{equation}
has rank exactly $m$; note that $m\le t$, and that $m\le n$. This open set $U$ is nonempty hence dense, for it contains the matrix in which the columns indexed~\eqref{equation:pfaffian:column:set} are, respectively, the standard basis vectors
\[
e_{t+2},\ e_{t+3},\ \dots,\ e_{t+m},\ e_{t+1},
\]
and all other columns are zero; the order of the standard basis vectors above accounts for the possibility that $V$ may be~$V(I_\sigma+J'_{n-1})$, though it cannot be $V(I_\sigma+J'_n)$, given our hypotheses. It suffices to compute the dimension of $U$.

Given a matrix $A$ in the open set $U$, let $B$ denote the $2t \times m$ submatrix consisting of the columns indexed~\eqref{equation:pfaffian:column:set}. For each $j$ with $1\le j\le m$, the submatrix $D_j$ of $A$ consisting of the columns indexed
$s_{j-1}+1,\dots,s_j$ can be uniquely written as a linear combination of the columns of $B|_j$. The coefficients needed comprise the columns of a size $j\times (s_j-s_{j-1})$ matrix that we denote $C_j$. The first column of $C_j$ is
\[
(0,\ 0,\ \dots,\ 0,\ 1)^\tr
\]
while the other $j(s_j-s_{j-1}-1)$ entries are arbitrary scalars. In the case $V(I_\sigma+J_a)$, the matrices~$B$ vary in a space of dimension
\[
2mt-\binom{m}{2}-k
\]
by Corollary~\ref{corollary:pfaffian:ci}, and it follows that $U$ has dimension
\begin{multline*}
2mt-\binom{m}{2}-k+1(s_1-s_0-1)+2(s_2-s_1-1)+\dots+m(s_m-s_{m-1}-1) \\ 
=\ m(2t+n-m)-k-\sum_{j=1}^{m-1}s_j.
\end{multline*}
The dimension count for $V(I_\sigma+J'_a)$ is similar, bearing in mind that in this case the matrices~$B$ vary in a space of dimension
\[
2mt-\binom{m}{2}-m-k.
\]

The proof of the Cohen-Macaulay property is again via induction, assuming the result for matrices $Y$ of smaller size, as well as for larger ideals in the family. Consider a prime of the form $I_\sigma+J_{s_k}$ where $k\le m-1$. Since $y_{1,s_k+1}$ is a nonzerodivisor on $S/(I_\sigma+J_{s_k})$, it suffices to prove that
\[
S/(I_\sigma+J_{s_k}+y_{1,s_k+1}S)\ =\ S/(I_\sigma+J_{s_k+1})
\]
is Cohen-Macaulay. If $s_k+1=s_{k+1}$, then this is immediate from the inductive hypothesis. Else, $s_k+1<s_{k+1}$, and we claim that $I_\sigma+J_{s_k+1}$ has minimal primes
\[
Q_1\colonequals I_\sigma+J_{s_{k+1}}\quad\text{ and }\quad Q_2\colonequals I_{\sigma'}+J_{s_k+1},
\]
where
\[
\sigma'\colonequals(s_0,s_1,\dots,s_{k-1},s_k+1,s_{k+1},\dots,s_m);
\]
if $k=0$, then~$\sigma'=(1,s_1,\dots,s_m)$ is not standard, but $Q_2$ is prime by the case of a matrix of size $2t\times(n-1)$, and the dimension of~$S/Q_2$ is readily computed. Since $I_\sigma+J_{s_k+1}$ is radical and contained in each $Q_i$, it suffices to verify that
\[
Q_1Q_2\ \subseteq\ I_\sigma+J_{s_k+1}.
\]
This is straightforward, since
\[
y_{1b}I_{k+1}(Y|_{s_k+1})\ \subseteq\ I_{k+2}(Y|_{s_{k+1}}) + J_{s_k+1}\ \subseteq\ I_\sigma+J_{s_k+1}
\]
for each $b$ with $b\le s_{k+1}$ by Lemma~\ref{lemma:minors}. By the inductive hypothesis, each $Q_i$ is prime, defining a Cohen-Macaulay ring $S/Q_i$. Moreover,
\[
Q_1+Q_2\ =\ I_{\sigma'}+J_{s_{k+1}}
\]
is prime, and Lemma~\ref{lemma:CM} applies since
\[
\dim S/Q_1\ =\ \dim S/Q_2\ =\ m(2t+n-m)-k-1-\sum_{j=1}^{m-1}s_j\ =\ \dim S/(Q_1+Q_2)+1.
\]
This concludes the argument that
\[
S/(I_\sigma+J_{s_k+1})\ =\ S/(Q_1\cap Q_2)
\]
is Cohen-Macaulay. The proof for a prime ideal of the form $I_\sigma+J'_{s_k}$, with $k\le m-1$, is similar, and left to the reader.

The remaining case is a prime of the form $I_\sigma+J_n$, where it suffices to prove that
\[
S/(I_\sigma+J_n+y_{t+1,1}S)\ =\ S/(I_\sigma+J'_1)
\]
is Cohen-Macaulay. This follows from the inductive hypothesis if $s_1=1$. If $s_1>1$, we claim that $I_\sigma+J'_1$ has minimal primes
\[
Q_1\colonequals I_\sigma+J'_{s_1}\quad\text{ and }\quad Q_2\colonequals I_{\sigma}+J_n+(y_{21},\ y_{31},\ \dots,\ y_{2t,1})S.
\]
For this, it suffices to verify that $Q_1Q_2\subseteq I_\sigma+J'_1$, which follows using $I_2(Y|_{s_1})\subseteq I_\sigma$. Note that $S/Q_2$ and $S/(Q_1+Q_2)$ are Cohen-Macaulay using the case of a smaller matrix, namely the matrix with the first column of $Y$ deleted. Since
\[
\dim S/Q_1\ =\ \dim S/Q_2\ =\ m(2t+n-m-1)-1-\sum_{j=1}^{m-1}s_j\ =\ \dim S/(Q_1+Q_2)+1,
\]
Lemma~\ref{lemma:CM} allows us to conclude that
\[
S/(I_\sigma+J'_1)\ =\ S/(Q_1\cap Q_2)
\]
is Cohen-Macaulay.
\end{proof}

We single out the main case of the previous theorem:

\begin{theorem}
\label{theorem:pfaffian:nullcone:main}
Let $Y$ be a~$2t\times n$ matrix of indeterminates over a field $K$, where $t$ and $n$ are positive integers. Set~$S\colonequals K[Y]$ and~$\frakP\colonequals (Y^\tr\Omega Y)S$, i.e., $\frakP$ is the ideal generated by the entries of the matrix $Y^\tr\Omega Y$. Then $S/\frakP$ is a Cohen-Macaulay integral domain, and
\[
\dim S/\frakP=\begin{cases}
2nt -\displaystyle{\binom{n}{2}}& \text{if }\ n\le t+1,\\
nt+\displaystyle{\binom{t+1}{2}}& \text{if }\ n\ge t.\\
\end{cases}
\]
\end{theorem}

\begin{proof}
The formulae for the dimension coincide when $n$ equals $t$ or $t+1$.

If $n\le t$ take $\sigma=(0,1,2,\dots,n-1,n)$ in Theorem~\ref{theorem:pfaffian:nullcone}~\eqref{theorem:pfaffian:nullcone:3}, to obtain
\[
\dim S/\frakP\ =\ n(2t+n-n)-(1+2+\dots+(n-1))\ =\ 2nt -\displaystyle{\binom{n}{2}},
\]
while if $n>t$, take $\sigma=(0,1,2,\dots,t-1,n)$, in which case the theorem gives
\[
\dim S/\frakP\ =\ t(2t+n-t)-(1+2+\dots+t-1)\ =\ nt+\binom{t+1}{2},
\]
completing the proof.
\end{proof}

\subsection{The purity of the embedding}

Using this, we settle the $\Sp_{2t}(K)$ case of Theorem~\ref{theorem:main}:

\begin{theorem}
\label{theorem:alternating:purity}
Let $K$ be a field of positive characteristic. Fix positive integers $n$ and~$t$, and consider the inclusion $\phi\colon K[Y^\tr\Omega Y]\to K[Y]$ where $Y$ is a size $2t\times n$ matrix of indeterminates. Then $\phi$ is pure if and only if $n\le t+1$.
\end{theorem}

\begin{proof}
We claim that if the inclusion $\phi\colon K[Y^\tr\Omega Y]\to K[Y]$ is pure for fixed $(n,t)$, then purity holds as well for the inclusion of the $K$-algebras corresponding to $(n',t)$ with~$n'\le n$.

Set $Y'\colonequals Y|_{n'}$, i.e., $Y'$ is the submatrix consisting of the first $n'$ columns of~$Y$. Consider the $\NN$-grading on~$K[Y]$ where the indeterminates from $Y'$ have degree $0$, as does $K$, while the remaining indeterminates have degree~$1$. Then
\[
{K[Y^\tr\Omega Y]}_0\ =\ K[Y'^\tr\Omega Y'],
\]
so $K[Y'^\tr\Omega Y']$ is a pure subring of $K[Y]$. It follows that the composition
\[
K[Y'^\tr\Omega Y']\ \subseteq\ K[Y^\tr\Omega Y]\ \subseteq\ K[Y]
\]
is pure as well, but then so is $K[Y'^\tr\Omega Y']\subseteq K[Y']$.

Set $S\colonequals K[Y]$ and $R\colonequals K[Y^\tr\Omega Y]$. We next prove that $\phi$ is pure in the case $n=t+1$. In this case the ring $R$ is regular, with the upper triangular entries of $Y^\tr\Omega Y$ forming a regular homogeneous system of parameters for $R$. As $\dim R=\binom{n}{2}$, it suffices by Theorem~\ref{theorem:solid} to verify that the local cohomology module 
\[
H^{\binom{n}{2}}_{\frakm_R}(S)
\]
is nonzero, where $\frakm_R$ is the homogeneous maximal ideal of $R$. This is immediate from Theorem~\ref{theorem:pfaffian:nullcone:main}, which implies that $\frakm_RS$ is an ideal of height $\binom{n}{2}$.

It remains to prove that $\phi\colon R\to S$ is not pure if $n\ge t+2$. By the reduction step, this comes down to the case $n=t+2$. In this case, the ring $R=K[Y^\tr\Omega Y]$ is again regular, of dimension $\binom{n}{2}$, so by Theorem~\ref{theorem:solid} it suffices to verify the vanishing of $H^{\binom{n}{2}}_{\frakm_R}(S)$. This follows from Theorem~\ref{theorem:pfaffian:nullcone:main}, which implies that~$\frakm_RS$ is an ideal of height
\[
\binom{n}{2}-1,
\]
defining a Cohen-Macaulay ring $S/\frakm_RS$.
\end{proof}

\section{Symmetric determinantal rings}
\label{sec:sym}

Let $X$ be an $n\times n$ symmetric matrix of indeterminates over a field $K$. For $d$ a positive integer, the ring~$K[X]/I_{d+1}(X)$ is a Cohen-Macaulay normal domain of dimension
\[
\binom{n+1}{2} - \binom{n+1-d}{2},
\]
with the convention that $\binom{i}{j}=0$ if $i<j$. The Cohen-Macaulay property is due to Kutz~\cite{Kutz}. The ring~$K[X]/I_{d+1}(X)$ is regular precisely if $n\le d$; when that is not the case, it has class group $\ZZ/2$, and is Gorenstein precisely if $n\equiv d+1\mod 2$, \cite{Goto1, Goto2}.

Let $Y$ be a $d\times n$ matrix of indeterminates over a field $K$, and set~$S\colonequals K[Y]$. For $X$ as above, the entrywise map of matrices
\[
X\to Y^\tr Y
\]
induces a $K$-algebra isomorphism between $K[X]/I_{d+1}(X)$ and the subring $R\colonequals K[Y^\tr Y]$ of~$S$. Our goal in this section is to determine when the inclusion~$\phi\colon R\to S$ is pure. The orthogonal group~$\O_d(K)$ acts $K$-linearly on $S$ via
\[
M\colon Y\mapsto MY\qquad\text{ for }\ M\in\O_d(K).
\]
Since $M^\tr M$ equals the identity matrix for $M\in\O_d(K)$, it is immediate that the entries of~$Y^\tr Y$ are fixed by the group action; when the field $K$ is infinite, of characteristic other than two, the invariant ring is precisely the subring~$R$, see~\cite[\S5]{DeConcini-Procesi}. When $K$ is an infinite field of characteristic two, the invariant ring is
\[
K[Y^\tr Y,\ \sum_{i=1}^d y_{ij}\ |\ 1\le j\le n],
\]
as proved by Richman~\cite[\S5]{Richman}; this corrects an error in~\cite[pp.~353--354]{DeConcini-Procesi}. A presentation for the invariant ring in this case is provided by~\cite[Proposition~23]{Richman}.

If $K$ has characteristic zero, then~$\O_d(K)$ is linearly reductive, and it follows that the invariant ring $R$ is a direct summand of $S$ as an $R$-module; specifically, $\phi\colon R\to S$ is pure when~$K$ has characteristic zero.

\subsection{The complete intersection property}

We work out the analogue of Theorem~\ref{theorem:pfaffian:ci} in the symmetric case. The ideal $(Y^\tr Y)$ has
\[
\binom{n+1}{2}
\]
minimal generators, coming from the distinct entries of the symmetric matrix $Y^\tr Y$. We next prove that in the case $n\le (d+1)/2$, these generators form a regular sequence, i.e., that~$K[Y]/(Y^\tr Y)$ is a complete intersection ring. More generally:

\begin{theorem}
\label{theorem:symmetric:ci}
Let $Y$ be a~$d\times n$ matrix of indeterminates over a field $K$, where $d$ and $n$ are positive integers with $n\le (d+1)/2$. For $k<n$, let $\fraka$ be an ideal generated by $k$ distinct entries from the first row. Then
\[
\dim K[Y]/((Y^\tr Y)+\fraka)\ =\ dn-\binom{n+1}{2}-k,
\]
i.e., $K[Y]/((Y^\tr Y)+\fraka)$ is a complete intersection ring.
\end{theorem}

\begin{proof}
It suffices to prove the assertion after specializing the entries of the last $d-2n+1$ rows to zero. We may hence assume that $n=(d+1)/2$, i.e., that $Y$ is size $(2n-1)\times n$.

First suppose $k=0$. Specialize the entries of $Y$ to the corresponding entries of the matrix
\[
\bar{Y}\colonequals\begin{pmatrix}
y_{11} & 0 & 0 & 0 & \cdots & 0 & 0 & 0\\
y_{21} & y_{21} & 0 & 0 & \cdots & 0 & 0 & 0\\
y_{31} & y_{32} & y_{31} & 0 & \cdots & 0 & 0 & 0\\
y_{41} & y_{42} & y_{42} & y_{41} & & 0 & 0 & 0\\
\vdots & \vdots & \vdots & \vdots & & & & \vdots\\
y_{n-1,1} & y_{n-1,2} & y_{n-1,3} & y_{n-1,4} & \cdots & y_{n-1,2} & y_{n-1,1} & 0\\
\hline
y_{n1} & y_{n2} & y_{n3} & y_{n4} & \cdots & y_{n3} & y_{n2} & y_{n1}\\
\hline
0 & y_{n+1,2} & y_{n+1,3} & y_{n+1,4} & \cdots & y_{n+1,4} & y_{n+1,3} & y_{n+1,2}\\
0 & 0 & y_{n+2,3} & y_{n+2,4} & \cdots & y_{n+2,5} & y_{n+2,4} & y_{n+2,3}\\
\vdots & \vdots & & & & \vdots & \vdots & \vdots\\
0 & 0 & 0 & 0 & & y_{2n-3,n-2} & y_{2n-3,n-1} & y_{2n-3,n-2}\\
0 & 0 & 0 & 0 & \cdots & 0 & y_{2n-2,n-1} & y_{2n-2,n-1}\\
0 & 0 & 0 & 0 & \cdots & 0 & 0 & y_{2n-1,n}
\end{pmatrix}.
\]
A routine---albeit tedious---count shows that this specialization entails killing
\[
3n(n-1)/2
\]
linear forms in $K[Y]$. The ideal $(Y^\tr Y)$ has~$\binom{n+1}{2}$ minimal generators; since
\[
\dim K[Y]\ =\ (2n-1)n\ =\ \binom{n+1}{2} + \frac{3}{2}n(n-1),
\]
it suffices to verify that
\[
K[\bar{Y}]/(\bar{Y}^\tr\bar{Y})
\]
has dimension zero. The $(1,n)$ entry of the matrix~$\bar{Y}^\tr\bar{Y}$ is $y_{n1}^2$. Modulo $y_{n1}$, the $(2,n)$ entry is~$y_{n+1,2}^2$. Proceeding in this order, examining the last column of $\bar{Y}^\tr\bar{Y}$, we see that
\[
y_{n1},\ y_{n+1,2},\ y_{n+2,3},\ \dots,\ y_{2n-2,n-1},\ y_{2n-1,n}
\]
are nilpotent in $K[\bar{Y}]/(\bar{Y}^\tr\bar{Y})$. Modulo these elements, the last column and the last two rows of~$\bar{Y}$ are zero; proceed inductively.

Since the displayed specialization $\bar{Y}$ entails killing $n-1$ entries from the first row, the case~$0<k<n$ follows as well.
\end{proof}

\subsection{Nullcones of symmetric determinantal rings in characteristic two}

Let $Y$ be a matrix of indeterminates of size~$d\times n$, over a field $K$ of characteristic two. The diagonal entries of the product matrix $Y^\tr Y$ are
\[
y_{11}^2+\dots+y_{d1}^2,\ \dots,\ y_{1n}^2+\dots+y_{dn}^2.
\]
Working in the ring $S\colonequals K[Y]$, the ideal
\[
\frakS\colonequals (Y^\tr Y)S+(y_{11}+\dots+y_{d1},\ \dots,\ y_{1n}+\dots+y_{dn})S
\]
agrees with $(Y^\tr Y)S$ up to radical; we prove next that $\frakS$ is a prime ideal, defining a Cohen-Macaulay ring:

\begin{theorem}
\label{theorem:symmetric:nullcone:char2}
Let $Y$ be a~$d\times n$ matrix of indeterminates over a field $K$ of characteristic two. Set~$S\colonequals K[Y]$ and let~$\frakS$ be as above. Write $d$ as $2t+1$ or $2t+2$, where $t$ is a nonnegative integer. Then $S/\frakS$ is a Cohen-Macaulay integral domain, and
\[
\dim S/\frakS=\begin{cases}
nd -\displaystyle{\binom{n+1}{2}}& \text{if }\ n\le t+1,\\
nt+\displaystyle{\binom{t+1}{2}}& \text{if $d=2t+1$ and }\ n\ge t,\\
n(t+1)+\displaystyle{\binom{t+1}{2}}& \text{if $d=2t+2$ and }\ n\ge t.
\end{cases}
\]
\end{theorem}

\begin{proof}
Let $\tilde{Y}$ denote the upper $(d-1)\times n$ submatrix of $Y$. In the ring $S/\frakS$ one has
\[
y_{di}\ =\ y_{1i}+\dots+y_{d-1,i}
\]
for each $i$, so $S/\frakS$ is a homomorphic image of $K[\tilde{Y}]$. Making the substitutions using the equation displayed above, one sees that
\[
S/\frakS\ \cong\ K[\tilde{Y}]/(\tilde{Y}^\tr\Psi\tilde{Y}),
\]
where $\Psi$ is the $(d-1)\times(d-1)$ alternating matrix
\[
\Psi\ =\
\begin{pmatrix}
0 & 1 & 1 & \hdots & 1\\
1 & 0 & 1 & \hdots & 1\\
1 & 1 & 0 & & 1\\
\vdots & \vdots & & \ddots & \\
1 & 1 & 1 & & 0
\end{pmatrix}.
\]
It is readily checked that $\Psi$ is invertible if $d-1$ is even, and that it has rank $d-2$ otherwise. Since alternating matrices of the same size are cogredient as in~\eqref{equation:cogredient} precisely if they have the same rank, if~$d-1$ is even then $\Psi$ is cogredient to the standard symplectic block matrix~$\Omega$, whereas, if $d-1$ is odd, then $\Psi$ is cogredient to
\[
\begin{pmatrix}
\Omega & 0\\
0 & 0
\end{pmatrix},
\]
where $\Omega$ is size $d-2$. This largely reduces the proof to an application of Theorem~\ref{theorem:pfaffian:nullcone:main}:

If $d=2t+1$, the ring $S/\frakS$ is isomorphic to $K[Z]/(Z^\tr\Omega Z)$, where $Z$ is a $2t\times n$ matrix of indeterminates. It follows that $S/\frakS$ is a Cohen-Macaulay integral domain, with
\[
\dim S/\frakS=\begin{cases}
2nt -\displaystyle{\binom{n}{2}} \ =\ nd -\displaystyle{\binom{n+1}{2}}& \text{if }\ n\le t+1,\\
nt+\displaystyle{\binom{t+1}{2}}& \text{if }\ n\ge t.
\end{cases}
\]

If $d=2t+2$, then $S/\frakS$ is isomorphic to a polynomial ring in $n$ indeterminates over the ring $K[Z]/(Z^\tr\Omega Z)$, where $Z$ is a matrix of indeterminates of size $2t\times n$. It follows that~$S/\frakS$ is again a Cohen-Macaulay integral domain, and that
\[
\dim S/\frakS=\begin{cases}
n+2nt-\displaystyle{\binom{n}{2}}\ =\ nd -\displaystyle{\binom{n+1}{2}}& \text{if }\ n\le t+1,\\
n+nt+\displaystyle{\binom{t+1}{2}}& \text{if }\ n\ge t,
\end{cases}
\]
which completes the proof.
\end{proof}

\subsection{Nullcones of symmetric determinantal rings in characteristic other than two}
\label{ssec:sym:odd}

Throughout this section, $K$ will denote a field of characteristic other than two. We study the nullcone $K[Y]/(Y^\tr Y)$, where $Y$ is a matrix of indeterminates of size $d\times n$.

Let $V$ be the vector space $K^d$ with the standard basis. Consider the symmetric bilinear form $B\colon V\times V\to K$ given by
\begin{equation}
\label{equation:symmetric:form}
(v_1,v_2)\mapsto v_1^\tr v_2.
\end{equation}
A subspace $W$ of $V$ is \emph{isotropic} if $B(w_1,w_2)=0$ for all $w_i\in W$. Since $B$ is nondegenerate, an isotropic subspace $W$ has rank at most $d/2$, where $V$ has rank $d$.

Let $M$ be a size $d\times n$ matrix over $K$ with $M^\tr M = 0$. Then the columns of $M$ span an isotropic subspace, so $\rank M \le d/2$. Setting $t\colonequals\lfloor d/2\rfloor$, it follows that $I_{t+1}(M)=0$. If $K$ is algebraically closed, the Nullstellensatz implies that
\[
I_{t+1}(Y)\ \subseteq\ \rad (Y^\tr Y)
\]
in the polynomial ring $K[Y]$. In view of this, set
\[
\frakS\colonequals (Y^\tr Y) + I_{t+1}(Y).
\]
When the size of $Y$ needs to be referenced, we use the notation $\frakS_{d\times n}$. When $d$ is odd, we shall prove that the ideal $\frakS$ is prime, and defines a Cohen-Macaulay ring~$K[Y]/\frakS$. When~$d$ is even with $d\le 2n$, it turns out that $\frakS$ has minimal primes $\frakP$ and~$\frakQ$, see Definition~\ref{definition:p:q}, with the rings $K[Y]/\frakP$ and $K[Y]/\frakQ$ being Cohen-Macaulay. All of this will be proved using principal radical systems.

The proof of the following is much the same as that of Lemma~\ref{lemma:functional:symplectic}:

\begin{lemma}
\label{lemma:functional:symmetric}
Let $K$ be a field. Consider the vector space $K^d$ equipped with a nondegenerate symmetric bilinear form. Let $L$ be a nonzero linear functional on $K^d$, and let
\[
V_1\subseteq V_2\subseteq \dots\subseteq V_m
\]
be isotropic subspaces of $K^d$ with $\rank V_j\le j$ for each $j$, where $m\le \lfloor d/2\rfloor$.

Suppose $L$ vanishes on $V_k$ \,for some $k$. Then there exist isotropic subspaces
\[
W_1\subset W_2\subset\dots\subset W_m
\]
such that, for each $j$, one has $V_j\subseteq W_j$ and $\rank W_j=j$, and $L$ vanishes on $W_k$.
\end{lemma}

\begin{remark}
\label{remark:orthogonal}
Let $K$ be an algebraically closed field of characteristic other than two. The orthogonal group~$\O_n(K)$ is the group of $n\times n$ matrices~$M$ over~$K$ with $M^\tr M=\1$. It follows that~$\O_n(K)$ is an algebraic group; it has two connected components, the special orthogonal group $\SO_n(K)$ consisting of elements with determinant~$1$, and its complement consisting of orthogonal matrices of determinant $-1$.

Let $W$ be an $n\times n$ matrix of indeterminates over $K$, in which case~$\O_n(K)$ may be viewed as the algebraic set~$V(W^\tr W-\1)$. The ideal $(W^\tr W-\1)$ is radical in $K[W]$, minimally generated by~$\binom{n+1}{2}$ polynomials that form a regular sequence, see for example~\cite[page~238]{Procesi}. Since $\O_n(K)$ is nonsingular, being an algebraic group, each irreducible component is nonsingular. By Serre's criterion, $K[W]/(W^\tr W-\1)$ is a normal ring; it is a product of normal domains corresponding to the two connected components.

For an integer $k$ with $k<n$, let $Z\colonequals W|_k$ denote the submatrix consisting of the first~$k$ columns of $W$. A minimal generating set for the ideal $(Z^\tr Z-\1)$ extends to one for the ideal $(W^\tr W-\1)$, so~$K[Z]/(Z^\tr Z-\1)$ is also a normal complete intersection ring. The map
\[
\SO_n(K)\ \to\ V(Z^\tr Z-\1)
\]
given by truncating columns is surjective since each matrix in $V(Z^\tr Z-\1)$ can be extended to one in $\SO_n(K)$. Since $\SO_n(K)$ is irreducible, so is its image; it follows that
\[
K[Z]/(Z^\tr Z-\1)
\]
is a normal \emph{domain}.
\end{remark}

\begin{definition}
\label{definition:sgn}
Let $\alpha$ be a subset of $\{1,\dots,n\}$, and $\alpha^\comp$ its complement. Set $\sgn(\alpha)$ to be the sign of the permutation that sends the $n$-tuple $(1,\dots,n)$ to the $n$-tuple $(\alpha,\alpha^\comp)$, where the entries of each of~$\alpha$ and~$\alpha^\comp$ are in ascending order.
\end{definition}

For a matrix $M$, a subset $\alpha$ of the row indices, and a subset $\beta$ of the column indices, set~$M_{\alpha|\beta}$ to be the submatrix with rows $\alpha$ and columns $\beta$. The following lemma appears to be well-known, but we include a proof based on \cite{Jagy}:

\begin{lemma}
\label{lemma:submatrix}
Let $Q\in\O_n(K)$. Let~$\alpha$ and $\beta$ be subsets of $\{1,\dots,n\}$ of cardinality $k$, where~$1\le k \le n-1$. Then
\[
\det(Q_{\alpha|\beta})\ =\ \sgn(\alpha) \sgn(\beta)\det(Q)\det(Q_{\alpha^\comp|\beta^\comp}).
\]
\end{lemma}

\begin{proof}
First consider the case $\alpha=\{1,\dots,k\}=\beta$. Let
\[
Q\ =\ \begin{pmatrix}
A & B\\
C & D
\end{pmatrix},
\]
where $A$ is a square matrix of size $k$, and $D$ is a square matrix of size $n-k$. Then
\[
\begin{pmatrix}
\1_k & 0 \\
0 & \1_{n-k} \end{pmatrix}
\ =\ QQ^\tr\ =\
\begin{pmatrix}
A & B\\
C & D
\end{pmatrix}
\begin{pmatrix}
A^\tr & C^\tr\\
B^\tr & D^\tr
\end{pmatrix}
\ =\
\begin{pmatrix}
AA^\tr + BB^\tr & AC^\tr + BD^\tr\\
CA^\tr + DB^\tr & CC^\tr+DD^\tr
\end{pmatrix},
\]
using which one has
\[
\begin{pmatrix}
A & B\\
C & D
\end{pmatrix}
\begin{pmatrix}
\1_k & C^\tr \\
0 & D^\tr
\end{pmatrix}
\ =\
\begin{pmatrix}
A & AC^\tr + BD^\tr \\
C & CC^\tr+DD^\tr
\end{pmatrix}
\ =\
\begin{pmatrix}
A & 0 \\
C & \1_{n-k}
\end{pmatrix}.
\]
Taking determinants gives
\[
\det Q\det D\ =\ \det A,
\]
which is precisely the assertion of the lemma in this case.

For arbitrary $\alpha,\beta$, permute the rows of $Q$ by sending the rows indexed $(\alpha,\alpha^\comp)$ to the rows indexed $(1,\dots,n)$, and the columns indexed $(\beta,\beta^\comp)$ to the columns indexed $(1,\dots,n)$. This yields an orthogonal matrix with determinant $\sgn(\alpha)\sgn(\beta)\det(Q)$. The result now follows from the previous case.
\end{proof}

\begin{definition}
\label{definition:p:q}
Let $Y$ be a $2t \times n$ matrix of indeterminates over a field $K$ of characteristic other than two, where $t\le n$. Assume that $K$ contains an element $i$ with $i^2=-1$.

Set $\frakP$ to be the ideal of $K[Y]$ generated by $\frakS$ and the polynomials
\[
\det(Y_{\alpha|\beta}) - i\,^t\sgn(\alpha)\det(Y_{\alpha^\comp|\beta}),
\]
for all subsets $\alpha\subseteq\{1,\dots,2t\}$ and $\beta\subseteq\{1,\dots,n\}$ of size $t$.

Similarly, set $\frakQ$ to be the ideal generated by $\frakS$ and the polynomials
\[
\det(Y_{\alpha|\beta}) + i\,^t\sgn(\alpha)\det(Y_{\alpha^\comp|\beta}),
\]
for all $\alpha$ and $\beta$ as before. We use~$\frakP_{2t\times n}$ and $\frakQ_{2t\times n}$ when the size of $Y$ needs clarification.
\end{definition}

It is readily seen that
\begin{equation}
\label{equation:p:q:in:I_t}
\frakP\ \subseteq\ (Y^\tr Y) + I_t(Y)
\quad\text{ and }\quad
\frakQ\ \subseteq\ (Y^\tr Y) + I_t(Y)
\end{equation}
in $K[Y]$, and that setting $J_n\colonequals(y_{11},\ y_{12},\ \dots,\ y_{1n})$ one has
\begin{equation}
\label{equation:p:q:containment}
\frakP+J_n\ =\ (Y^\tr Y) + I_t(Y) + J_n\ =\ \frakQ+J_n.
\end{equation}

\begin{lemma}
\label{lemma:alpha}
Suppose $M$ and $Q$ are $n\times n$ matrices over a field $K$, where $Q\in\O_n(K)$. Let~$\alpha$ be a size $n$ subset of $\{1,\dots,2n\}$, and $\alpha^\comp$ its complement. Then
\[
\det\left[\begin{pmatrix} M \\ -iQM \end{pmatrix}_{\alpha}\right]\ =\ 
i^n \sgn(\alpha)(\det Q)\det\left[\begin{pmatrix} M \\ -iQM \end{pmatrix}_{\alpha^\comp}\right],
\]
where $(\phantom{M})_{\alpha}$ denotes the submatrix with rows $\alpha$, and $\sgn(\alpha)$ is as in Definition~\ref{definition:sgn}.
\end{lemma}

\begin{proof}
Using $\begin{pmatrix} M \\ -iQM \end{pmatrix}=\begin{pmatrix} \1_n \\ -iQ \end{pmatrix}M$, it suffices to prove the result when $M$ is the identity matrix. First consider the case
\[
\alpha\colonequals\{1,2,\dots,k,\ n+k+1,n+k+2,\dots,2n\},
\]
and write
\[
\begin{pmatrix} \1_n \\ -iQ \end{pmatrix}\ =\ \begin{pmatrix}
\1_k & 0\\
0 & \1_{n-k}\\
-iA & -iB\\
-iC & -iD
\end{pmatrix},
\]
where $Q=\begin{pmatrix}A & B\\C & D\end{pmatrix}$ for square matrices $A$ and $D$ of size $k$ and $n-k$ respectively. Then
\[
\det\left[\begin{pmatrix}\1_n\\ -iQ\end{pmatrix}_{\alpha}\right]\ =\ 
\det\begin{pmatrix}
\1_k & 0\\
-iC & -iD
\end{pmatrix}\ =\ (-i)^{n-k}\det D
\]
and
\begin{multline*}
\det\left[\begin{pmatrix}\1_n\\ -iQ\end{pmatrix}_{\alpha^\comp}\right]\ =\ 
\det\begin{pmatrix}
0 & \1_{n-k}\\
-iA & -iB
\end{pmatrix}\ =\ (-1)^{k(n-k)} \det\begin{pmatrix}
\1_{n-k} & 0\\
-iB & -iA
\end{pmatrix}\\
=\ (-1)^{k(n-k)}(-i)^k\det A.
\end{multline*}
The required verification is now
\[
(-i)^{n-k}\det D\ =\ i^n\sgn(\alpha)(\det Q) (-1)^{k(n-k)}(-i)^k\det A,
\]
which follows since $\sgn(\alpha)=(-1)^{n(n-k)}$ and $\det Q\det D=\det A$ by Lemma~\ref{lemma:submatrix}.

For an arbitrary $\alpha$, permute rows and columns, keeping track of sign changes, so as to reduce to the case settled above.
\end{proof}

The following proposition is the analogue of Proposition~\ref{proposition:irreducible:pfaffian} in the symmetric case:

\begin{proposition}
\label{proposition:symmetric:irreducible}
Let $d$ and $n$ be positive integers with $n\le d/2$. Let $Y$ be a $d\times n$ matrix of indeterminates over an algebraically closed field $K$ of characteristic other than two. For $a$ an integer with $0\le a \le n$, set
\[
J_a\colonequals(y_{11},\ y_{12},\ \dots,\ y_{1a})S
\]
and $I\colonequals (Y^\tr Y)S + J_a$, where $S\colonequals K[Y]$.
\begin{enumerate}[\quad \rm(1)]
\item If $n<d/2$ and $a<n$, then $V(I)$ is irreducible.
\item If $n=d/2$ and $a<n$, then $V(I)$ has irreducible components $V(\frakP+J_a)$, $V(\frakQ+J_a)$.
\end{enumerate}
\end{proposition}

\begin{proof}
Let $\Delta$ be the upper $n\times n$ minor of $Y$. We claim that $\Delta$ is a nonzerodivisor on $S/I$. Since $S/I$ is a complete intersection ring by Theorem~\ref{theorem:symmetric:ci}, it suffices to show that $\Delta$ does not belong to any minimal prime of $I$.

Let $G$ be a copy of $\SO_{d-1}(K)$, embedded in $\SO_d(K)$ as
\[
\begin{pmatrix}
1 & 0\\
0 & Q
\end{pmatrix},
\qquad\text{ for }\ Q\in\SO_{d-1}(K).
\]
The action of $G$ on $S$ with $M\colon Y\mapsto MY$ induces an action on $S/I$, and hence on the set of minimal primes of $S/I$. Since $G$ is connected, this action must be trivial, i.e., $G$ stabilizes each minimal prime of $S/I$.

Up to sign changes, rows of $Y$ other than the first row may be permuted using an element of~$G$. It follows that under the action of $G$ on $S$, each maximal minor of $Y$ that involves the first row is in the orbit of $\Delta$, so any minimal prime of $I$ containing $\Delta$ also contains each maximal minor involving the first row; said otherwise, if $\Delta$ vanishes on an irreducible component of~$V(I)$, then so does each such minor.

For a $d\times n$ matrix over $K$, if each maximal minor that involves the first row is zero, and some other maximal minor is nonzero, then the first row must be zero. Hence if $\Delta$ vanishes on some irreducible component of $V(I)$, then either $J_n$ or $I_n(Y)$ vanishes on that component; in other words, any minimal prime of $I$ containing $\Delta$ must contain either $J_n$ or each maximal minor of $Y$. Since $a<n$, one has
\[
\dim V(I)\ =\ dn-\binom{n+1}{2}-a\ >\ \dim V(I + J_n)\ =\ (d-1)n-\binom{n+1}{2},
\]
so no minimal prime of $I$ contains $J_n$. It follows that any minimal prime of $I$ that contains~$\Delta$ also contains $I_n(Y)$.

Let $Y'$ be the submatrix consisting of the first $n-1$ columns of $Y$, and consider the ideal
\[
I'\colonequals (Y'^\tr Y')+ (y_{11},\ y_{12},\ \dots,\ y_{1a})
\]
of $K[Y']$. Viewing a point of $V(I')$ as columns $(v_1,\dots,v_{n-1})$, the image of the map 
\begin{alignat*}2
V(I') & \times K^{n-1}\ && \to\ V(I+I_n(Y))\\
((v_1,\dots,v_{n-1})&\ \,,\ (c_1,\dots,c_{n-1}))\ && \mapsto\ (v_1,\dots,v_{n-1},\sum c_iv_i)
\end{alignat*}
includes the open subset of $V(I+I_n(Y))$ where the first $n-1$ columns are linearly independent. Hence
\[
\dim V(I+I_n(Y))\ \le\ \dim V(I')+(n-1)\ =\ d(n-1)-\binom{n}{2} - a +(n-1)\ <\ \dim V(I).
\]
It follows that a minimal prime of $I$ cannot contain $I_n(Y)$, and hence that $\Delta$ is not in any minimal prime of $I$. The completes the proof that $\Delta$ is a nonzerodivisor on $S/I$. In light of this, there is a bijection between the minimal primes of $S/I$ and those of $S_\Delta/I$.

Write the matrix $Y$ as $\begin{pmatrix}Y_1\\ Y_2\end{pmatrix}$, where $Y_1$ is the upper $n\times n$ submatrix, so that $\Delta=\det Y_1$. Since~$Y_1$ is an invertible matrix over $S_\Delta$, one has
\[
S_\Delta\ =\ K[Y_1,\ Y_2]_\Delta \ =\ K[Y_1,\ Y_2Y_1^{-1}]_\Delta,
\]
so the entries of $Y_2Y_1^{-1}$, and hence of $Z\colonequals iY_2Y_1^{-1}$, are algebraically independent over the fraction field of $K[Y_1]$. Since
\[
YY_1^{-1}\ =\ \begin{pmatrix}Y_1\\ Y_2\end{pmatrix}Y_1^{-1}\ =\ \begin{pmatrix}\1\\ -iZ\end{pmatrix},
\]
the ideal $(Y^\tr Y)S_\Delta$ agrees with the ideal generated by the entries of
\[
(Y_1^{-1})^\tr Y^\tr YY_1^{-1}\ =\
\begin{pmatrix}\1&-iZ^\tr\end{pmatrix}\begin{pmatrix}\1\\ -iZ\end{pmatrix}\ =\ \1-Z^\tr Z,
\]
i.e., $S_\Delta/(Y^\tr Y)=K[Y_1,\ Z]_\Delta/(Z^\tr Z-\1)$. As $J_a$ is generated by indeterminates from the matrix $Y_1$, the minimal primes of $S_\Delta/I$ correspond to those of $K[Z]/(Z^\tr Z-\1)$, and it suffices to prove the theorem in the case $a=0$. 

If $n<d/2$, one has $n<d-n$, so $K[Z]/(Z^\tr Z-\1)$ is a domain by Remark~\ref{remark:orthogonal}, completing the proof of~(1). When $n=d/2$, the matrix $Z$ is $n\times n$, so $V(Y^\tr Y)$ has two irreducible components corresponding to the two components of $\O_n(K)=V(Z^\tr Z-\1)$, though it remains to verify that these are precisely~$V(\frakP)$ and $V(\frakQ)$.

The homomorphism
\[
K[Y]\ =\ K[Y_1,\ Y_2]\ \to\ K[Y_1,\ Z]/(Z^\tr Z-\1)
\]
with $Y_2\mapsto -iZY_1$ kills $(Y^\tr Y)$, giving a homomorphism
\[
K[Y]/(Y^\tr Y)\ \to\ K[Y_1,\ Z]/(Z^\tr Z-\1),
\]
that is an isomorphism upon inverting $\Delta$. Since $\Delta$ is nonzerodivisor in $K[Y]/(Y^\tr Y)$, the ideal $(Y^\tr Y)$ is radical. The homomorphism above gives a map
\begin{alignat*}2
\AA^{\!n^2} &\times \O_n(K) &\ \to & \ \ V(Y^\tr Y)\\
(A &\ \,,\, Q)\ & \mapsto & \ \begin{pmatrix}A\\ -iQA\end{pmatrix}.
\end{alignat*}
Using Lemma~\ref{lemma:alpha}, the matrix $\begin{pmatrix}A\\ -iQA\end{pmatrix}$ lies in the algebraic set $V(\frakP)$ if $Q\in\SO_n(K)$, and in the algebraic set $V(\frakQ)$ otherwise. Hence the map displayed above restricts to maps
\[
\AA^{\!n^2}\times\SO_n(K)\ \to\ V(\frakP)
\quad\text{ and }\quad
\AA^{\!n^2}\times\O_n(K)\smallsetminus\SO_n(K)\ \to\ V(\frakQ).
\]
Since $V(\frakP)\cup V(\frakQ)$ contains $V(Y^\tr Y)\smallsetminus V(\Delta)$, we have
\[
\frakP\cap\frakQ\ \subseteq\ (Y^\tr Y)K[Y]_\Delta.
\]
Using again that $\Delta$ is nonzerodivisor in $K[Y]/(Y^\tr Y)$, it follows that $\frakP\cap\frakQ=(Y^\tr Y)$.
\end{proof}

\begin{corollary}
\label{corollary:symmetric:components}
Let $Y$ be a $2t\times n$ matrix of indeterminates over an algebraically closed field of characteristic other than two. Then the algebraic set $V(Y^\tr Y)$ equals $V(\frakP)\cup V(\frakQ)$.
\end{corollary}

\begin{proof}
One containment is immediate as the ideals $\frakP$ and $\frakQ$ contain $(Y^\tr Y)$. Let $M$ be a matrix in $V(Y^\tr Y)$. If $M$ has rank less than $t$, then it belongs to each of $V(\frakP)$ and $V(\frakQ)$. In the remaining case, $M$ has rank exactly $t$; assume without loss of generality that the first $t$ columns of $M$ are linearly independent. Then the $2t\times t$ submatrix $M|_t$ belongs to $V(\frakP|_t)$ or $V(\frakQ|_t)$ by Proposition~\ref{proposition:symmetric:irreducible}. Since the remaining columns of $M$ are linear combinations of the columns of $M|_t$, it follows that $M$ belongs to $V(\frakP)$ or $V(\frakQ)$.
\end{proof}

We now set up the principal radical system needed to study the ideals $\frakP$, $\frakQ$, and $\frakS$. Let~$Y$ be a $d\times n$ matrix of indeterminates over $K$; recall that
\[
\frakS\colonequals (Y^\tr Y) + I_{t+1}(Y),
\]
where~$t\colonequals\lfloor d/2\rfloor$. Let $\sigma\colonequals(s_0,s_1,s_2,\dots,s_m)$ be a sequence of integers with $0\le s_k\le n$ for each $k$, and $s_m=n$. Set
\[
I_\sigma\colonequals\frakS + I_1\big(Y|_{s_0}\big) + I_2\big(Y|_{s_1}\big) + I_3\big(Y|_{s_2}\big) + \dots + I_{m+1}\big(Y|_{s_m}\big),
\]
where, as earlier, $I_{k+1}\big(Y|_{s_k}\big)$ denotes the ideal generated by the size $k+1$ minors of the submatrix consisting of the first $s_k$ columns of $Y$. If $d=2t$, set
\[
I'_\sigma\colonequals\frakP + I_1\big(Y|_{s_0}\big) + I_2\big(Y|_{s_1}\big) + I_3\big(Y|_{s_2}\big) + \dots + I_{m+1}\big(Y|_{s_m}\big)
\]
and
\[
I''_\sigma\colonequals\frakQ + I_1\big(Y|_{s_0}\big) + I_2\big(Y|_{s_1}\big) + I_3\big(Y|_{s_2}\big) + \dots + I_{m+1}\big(Y|_{s_m}\big).
\]
Note that if $m<t$, then both $I'_\sigma$ and $I''_\sigma$ contain $I_t(Y)$, and hence equal $I_\sigma$.

We say $\sigma$ is \emph{standard} if
\[
0=s_0 < s_1 < s_2 < \dots < s_m=n,\quad\text{ and }\quad m\le t.
\]
For integers $a$ with $0\le a\le n$, set
\[
J_a\colonequals (y_{11},\ y_{12},\ \dots,\ y_{1a}).
\]

Suppose $\sigma\colonequals(s_0,s_1,s_2,\dots,s_m)$ is standard, $d=2t$, $m=t$, and $s_{m-1}<a<s_m$. Define
\[
\sigma'\colonequals(s_0,s_1,s_2,\dots,s_{m-2},a,s_m).
\]
We claim that
\[
I'_\sigma+J_a\ =\ I'_{\sigma'}+J_a
\quad\text{ and }\quad
I''_\sigma+J_a\ =\ I''_{\sigma'}+J_a.
\]
For the first equality, it suffices to verify that 
\[
I_t(Y|_a)\ \subseteq\ \frakP + J_a,
\]
which holds by~\eqref{equation:p:q:containment} since the first row of $Y|_a$ is zero modulo $J_a$. The second is similar.

With the notation as above, we prove:

\begin{theorem}
\label{theorem:symmetric:prs}
Let $Y$ be a~$d\times n$ matrix of indeterminates over an algebraically closed field $K$ of characteristic other than two, and set $S\colonequals K[Y]$. Let $\sigma\colonequals(s_0,s_1,s_2,\dots,s_m)$ be a sequence of integers with $0\le s_k\le n$ for each $k$, and $s_m=n$. Fix $a$ with $0\le a\le n$.

\begin{enumerate}[\quad \rm(1)]
\item Suppose $\sigma$ is standard and $a=s_k$ where $0\le k\le m-1$. If $d$ is odd, then $V(I_\sigma + J_a)$ is irreducible; if $d$ is even, then $V(I'_\sigma + J_a)$ and $V(I''_\sigma + J_a)$ are irreducible.

\item The ideal $I_\sigma+J_a$ is radical; if $d$ is even, the ideals $I'_\sigma+J_a$ and $I''_\sigma+J_a$ are radical.

\item Suppose $\sigma$ is standard and $a=s_k$ where $0\le k\le m-1$. If $d$ is odd, then $I_\sigma+J_a$ defines a Cohen-Macaulay integral domain; if $d$ is even, $I'_\sigma + J_a$ and $I''_\sigma + J_a$ both define Cohen-Macaulay integral domains. In each case, the domain has dimension
\[
m(d+n-m-1)-k-\sum_{j=1}^{m-1}s_j.
\]
\end{enumerate}
\end{theorem}

\begin{proof}
Let $\VV$ denote one of the algebraic sets $V(I_\sigma+J_{s_k})$ or $V(I'_\sigma+J_{s_k})$ or $V(I''_\sigma+J_{s_k})$ under the hypotheses of~(1); we first prove that $\VV$ is irreducible. Take $\VV_0$ to be the set of~$d\times m$ matrices lying in either $V(\frakS_{d\times m})$ or $V(\frakP_{d\times m})$ or~$V(\frakQ_{d\times m})$, in the respective cases, with the additional condition that the first~$k$ entries of the first row are~$0$. Note that~$m\le d/2$ and~$k<m$, so $\VV_0$ is irreducible by Proposition~\ref{proposition:symmetric:irreducible}.

Let $B$ be an element of $\VV_0$. For $1\le j\le m$, let $C_j$ be a matrix of size $j\times(s_j-s_{j-1})$, and set $A$ to be the matrix
\begin{equation}
\label{equation:symmetric:irreducible}
(B|_1C_1)\,\#\,(B|_2C_2)\,\#\,\cdots\,\#\,(B|_mC_m),
\end{equation}
where $\#$ denotes concatenation; it is readily seen that $A$ is an element of the algebraic set $\VV$. The matrices $C_1,\dots,C_m$ may be regarded as the points of an affine space~$\VV_1$ of dimension
\[
\sum_{j=1}^m j(s_j-s_{j-1}),
\]
so that the construction~\eqref{equation:symmetric:irreducible} gives a map $\VV_0\times \VV_1\to\VV$. Since the image of an irreducible algebraic set is irreducible, it suffices to verify that this map is surjective.

Let $A$ be a matrix in the algebraic set $\VV$. For $1\le j\le m$, let $V_j$ denote the span of the columns of the truncated matrix $A|_{s_j}$. Consider the linear functional $L$ that is projection to the first coordinate, and the symmetric bilinear form is as defined in~\eqref{equation:symmetric:form}. By Lemma~\ref{lemma:functional:symmetric}, there exist isotropic subspaces
\[
W_1 \subset W_2 \subset \dots \subset W_m
\]
such that $V_j\subseteq W_j$ for each $j$, and $W_j$ has rank $j$. Consider a size $d\times m$ matrix $B$ such that~$B|_j$ spans $W_j$ for each $j$. Then the columns of $A|_{s_j}$ belong to the column span of $B|_j$ for each $j$, so there exist matrices $C_j$ using which $A$ may be obtained as in~\eqref{equation:symmetric:irreducible}. This concludes the proof of (1).

Set $I^*_\sigma$ to be one of $I_\sigma, I'_\sigma, I''_\sigma$, and $I\colonequals I^*_\sigma+J_a$. To show $I$ is radical or prime, we assume the result for matrices $Y$ of smaller size, as well as for larger ideals in the family, and apply Lemma~\ref{lemma:prs}. The three families are interlaced in the inductive process, since
\[
K[Y_{2t\times n}]/(\frakP_{2t\times n}+J_n)\ \cong\ K[Y_{2t-1\times n}]/\frakS_{2t-1\times n}
\ \cong\ K[Y_{2t\times n}]/(\frakQ_{2t\times n}+J_n)
\]
using~\eqref{equation:p:q:containment}, and Corollary~\ref{corollary:symmetric:components} gives
\[
K[Y_{2t+1\times n}]/\rad(\frakS_{2t+1\times n}+J_n)\ \cong\ K[Y_{2t\times n}]/\rad(\frakP_{2t\times n}\cap\frakQ_{2t\times n}).
\]

Assume $a<n$, since otherwise $K[Y]/I$ effectively involves a matrix of size $(d-1)\times n$. In applying Lemma~\ref{lemma:prs}, set
\[
x\colonequals y_{1,a+1}.
\]
Specializing $Y$ such that $y_{1,a+1}\mapsto 1$, and $y_{2,a+1}\mapsto \pm i$, and every other entry maps to~$0$, we obtain a matrix in $V(I)\smallsetminus V(I+xS)$; the choice of sign in $\pm i$ is relevant when $d=2$, and depends on whether $I$ contains $\frakP$ or $\frakQ$. It follows that $I+xS$ is a strictly larger ideal: in particular, for $a<n-1$ we have 
\[
I+xS\ =\ I^*_\sigma + J_{a+1},
\]
and for $a=n-1$ we have $I+xS = I^*_\sigma + J_n$, which effectively puts us in the case of a smaller matrix. Hence, in each case, $I+xS$ is radical by the inductive hypothesis. If $a$ is as in~(1), the ideal $P\colonequals \rad I$ is prime; since~$x\notin P$, Lemma~\ref{lemma:prs} implies that $I=P$, and hence that $I$ is prime. Else there exists an integer $k$ with~$s_k<a<s_{k+1}$. Set
\[
\sigma'\colonequals(s_0,s_1,\dots,s_{k-1},a,s_{k+1},\dots,s_m),
\]
and take $P$ to be the prime $I^*_{\sigma'}+J_a$; if $k=0$, then~$\sigma'=(a,s_1,\dots,s_m)$ is not standard, but the primality still holds from the case of a smaller matrix. The specialization used earlier shows that $x\notin P$. To conclude that $I$ is radical by Lemma~\ref{lemma:prs}, it remains to verify that~$xP\subseteq I$. For this, note that
\[
y_{1,a+1}I_{k+1}(Y|_a)\ \subseteq\ I_{k+2}(Y|_{a+1})+J_a\ \subseteq\ I_{k+2}(Y|_{s_{k+1}})+J_a\ \subseteq\ I,
\]
where the first inclusion is using Lemma~\ref{lemma:minors}.

For $a$ as in (3), we next compute the dimension of the algebraic set $\VV\colonequals V(I_\sigma+J_a)$. Consider the open subset $U$ of $\VV$ in which each matrix has the property that the submatrix consisting of the columns indexed
\begin{equation}
\label{equation:symmetric:column:set}
s_0+1,\ s_1+1,\ \dots,\ s_{m-1}+1
\end{equation}
has rank exactly $m$. This open set $U$ is nonempty hence dense, for it contains the matrix in which the columns indexed~\eqref{equation:symmetric:column:set} are the first $m$ columns of the matrix
\[
\begin{pmatrix}
0 & 0 & \hdots & 0\\
1 & 0 & \hdots & 0\\
i & 0 & \hdots & 0\\
0 & 1 & & 0\\
0 & i & & 0\\
\vdots & & & \\
0 & 0 & & 1\\
0 & 0 & & i
\end{pmatrix}
\quad\text{ or }\quad
\begin{pmatrix}
0 & 0 & \hdots & 0 & 1\\
0 & 0 & \hdots & 0 & i\\
1 & 0 & \hdots & 0 & 0\\
i & 0 & \hdots & 0 & 0\\
0 & 1 & & 0 & 0\\
0 & i & & 0 & 0\\
\vdots & & & & \vdots\\
0 & 0 & & 1 & 0\\
0 & 0 & & i & 0
\end{pmatrix},
\]
depending on whether $d$ is odd or even, respectively, and the remaining columns are zero.

It suffices to compute the dimension of $U$. Given a matrix $A$ in the $U$, let $B$ denote the~$d\times m$ submatrix consisting of the columns indexed~\eqref{equation:symmetric:column:set}. For each $j$ with $1\le j\le m$, the submatrix $D_j$ of $A$ consisting of the columns indexed $s_{j-1}+1,\dots,s_j$ can be uniquely written as a linear combination of the columns of $B|_j$. The coefficients needed comprise the columns of a size $j\times (s_j-s_{j-1})$ matrix that we denote $C_j$. The first column of $C_j$ is
\[
(0,\ 0,\ \dots,\ 0,\ 1)^\tr
\]
while the other $j(s_j-s_{j-1}-1)$ entries are arbitrary scalars. By Proposition~\ref{proposition:symmetric:irreducible}, the matrices~$B$ vary in a space of dimension
\[
dm-\binom{m+1}{2}-k,
\]
so $U$ has dimension
\begin{multline*}
dm-\binom{m+1}{2}-k+1(s_1-s_0-1)+2(s_2-s_1-1)+\dots+m(s_m-s_{m-1}-1) \\ 
=\ m(d+n-m-1)-k-\sum_{j=1}^{m-1}s_j.
\end{multline*}
It follows that $V(I_\sigma+J_a)$ has the dimension as claimed; when $d$ is even, $V(I'_\sigma + J_a)$ and~$V(I''_\sigma + J_a)$ also have the dimension displayed above.

The proof of the Cohen-Macaulay property is again via induction, assuming the result for smaller matrices and for larger ideals in the family. Consider first a prime ideal of the form $I_\sigma+J_a$, where $a\colonequals s_k<n$, and $d$ is odd. Since the element $y_{1,a+1}$ is a nonzerodivisor on~$S/(I_\sigma+J_a)$, it suffices to verify that
\[
S/(I_\sigma+J_a+y_{1,a+1}S)\ =\ S/(I_\sigma+J_{a+1})
\]
is a Cohen-Macaulay ring. The proof of this, for $d$ odd, is split into five cases:

Case (i): Suppose $k\le m-2$. If $a+1=s_{k+1}$, then~$S/(I_\sigma+J_{a+1})$ is Cohen-Macaulay by the inductive hypothesis. If $a+1<s_{k+1}$, we claim that $I_\sigma+J_{a+1}$ is the intersection of the prime ideals
\[
Q_1\colonequals I_\sigma+J_{s_{k+1}}\quad\text{ and }\quad Q_2\colonequals I_{\sigma'}+J_{a+1},
\]
where $\sigma'\colonequals(s_0,s_1,\dots,s_{k-1},a+1,s_{k+1},\dots,s_m)$; if $k=0$, then $Q_2$ is prime by the case of a matrix of size $d\times(n-1)$. Since $I_\sigma+J_{a+1}$ is radical and contained in each $Q_i$, it suffices to verify that
\[
Q_1Q_2\ \subseteq\ I_\sigma+J_{a+1},
\]
which comes down to
\[
J_{s_{k+1}}I_{\sigma'}\ \subseteq\ I_\sigma+J_{a+1}.
\]
This is straightforward, since for each $b$ with $b\le s_{k+1}$, one has
\[
y_{1b}I_{k+1}(Y|_{a+1})\ \subseteq\ I_{k+2}(Y|_{s_{k+1}}) + J_{a+1}\ \subseteq\ I_\sigma+J_{a+1}
\]
using Lemma~\ref{lemma:minors}. By the inductive hypothesis, each prime $Q_i$ defines a Cohen-Macaulay ring $S/Q_i$. Moreover,
\[
Q_1+Q_2\ =\ I_{\sigma'}+J_{s_{k+1}}
\]
is prime, and Lemma~\ref{lemma:CM} applies since
\[
\dim S/Q_1\ =\ \dim S/Q_2\ =\ m(d+n-m-1)-k-1-\sum_{j=1}^{m-1}s_j\ =\ \dim S/(Q_1+Q_2)+1.
\]
It follows that
\[
S/(I_\sigma+J_{a+1})\ =\ S/(Q_1\cap Q_2)
\]
is Cohen-Macaulay.

Case (ii): Next suppose $k=m-1$, and $m<\lfloor d/2\rfloor$, and $a+1=n$. Set $t\colonequals (d-1)/2$, and let $Y'$ denote the lower $2t\times n$ submatrix of~$Y$. Since $I_\sigma$ contains the ideal $I_{m+1}(Y)$ and hence $I_{t}(Y)$, it follows that $I_\sigma + J_n$ contains $\frakP_{2t\times n}(Y')$ and $\frakQ_{2t\times n}(Y')$ by~\eqref{equation:p:q:in:I_t}. But then
\[
S/(I_\sigma + J_n)\ =\ K[Y']/(I'_\sigma)\ =\ K[Y']/(I''_\sigma),
\]
which is Cohen-Macaulay by the case of a smaller matrix.

Case (iii): Suppose $k=m-1$, and $m<\lfloor d/2\rfloor$, and $a+1<n$. Then $I_\sigma+J_{a+1}$ is the intersection of the prime ideals
\[
Q_1\colonequals I_\sigma+J_n\quad\text{ and }\quad Q_2\colonequals I_{\sigma'}+J_{a+1},
\]
where $\sigma'\colonequals(s_0,s_1,\dots,s_{k-1},a+1,s_{k+1},\dots,s_m)$. The ring $S/Q_1$ is Cohen-Macaulay by the case of a smaller matrix, and the proof proceeds along the lines of~Case (i).

Case (iv): Suppose $k=m-1$, and $m=\lfloor d/2\rfloor$, and $a+1=n$. Then
\[
I_\sigma+J_{a+1}\ =\ I_\sigma+J_n
\]
is the intersection of the prime ideals
\[
Q_1\colonequals I_\sigma+J_n+\frakP\quad\text{ and }\quad Q_2\colonequals I_\sigma+J_n+\frakQ,
\]
where $\frakP\colonequals\frakP_{2m\times n}(Y')$ and $\frakQ\colonequals\frakQ_{2m\times n}(Y')$, with $Y'$ the lower $2m\times n$ submatrix of~$Y$. The rings $S/Q_1$ and $S/Q_2$ are Cohen-Macaulay by the inductive hypothesis, of dimension
\[
m(m+n-1)-\sum_{j=1}^{m-1}s_j.
\]
Moreover,
\[
Q_1+Q_2\ =\ I_\sigma+J_n+\frakP+\frakQ\ =\ I_\sigma+J_n+I_m(Y)\ =\ I_{\sigma'}+J_n,
\]
where $\sigma'\colonequals(s_0,s_1,\dots,s_{m-2},n)$, so $S/(Q_1+Q_2)$ is Cohen-Macaulay of dimension
\[
(m-1)(m+n)-\sum_{j=1}^{m-2}s_j.
\]
But then
\[
\dim S/Q_1\ =\ \dim S/Q_2\ =\ \dim S/(Q_1+Q_2)+1,
\]
so Lemma~\ref{lemma:CM} implies that $S/(Q_1\cap Q_2)$ is Cohen-Macaulay.

Case (v): Lastly, suppose $k=m-1$, and $m=\lfloor d/2\rfloor$, and $a+1<n$. We claim that the ideal $I_\sigma+J_{a+1}$ is the intersection of three prime ideals
\[
Q_1\colonequals I_\sigma+J_n+\frakP,\qquad Q_2\colonequals I_\sigma+J_n+\frakQ, \qquad Q_3\colonequals I_{\sigma'}+J_{a+1},
\]
where $\frakP\colonequals\frakP_{2m\times n}(Y')$ and $\frakQ\colonequals\frakQ_{2m\times n}(Y')$, with $Y'$ the lower $2m\times n$ submatrix of~$Y$,~and
\[
\sigma'\colonequals(s_0,s_1,\dots,s_{m-2},a+1,s_m).
\]
Since $Q_1\cap Q_2=I_\sigma+J_n$, and $I_\sigma+J_{a+1}$ is radical, it suffices to verify that
\[
(I_\sigma+J_n)(I_{\sigma'}+J_{a+1})\ \subseteq\ I_\sigma+J_{a+1},
\]
which is a now-routine application of Lemma~\ref{lemma:minors}.

Towards proving that $S/(Q_1\cap Q_2\cap Q_3)$ is Cohen-Macaulay, first note that
\[
Q_1+Q_3\ =\ I_{\sigma'}+J_n+\frakP\quad\text{ and }\quad Q_2+Q_3\ =\ I_{\sigma'}+J_n+\frakQ
\]
so the dimension formula proved earlier gives
\[
\dim S/Q_1\ =\ \dim S/Q_2\ =\ \dim S/Q_3\ =\ \dim S/(Q_1+Q_3)+1 =\ \dim S/(Q_2+Q_3)+1.
\]
Lemma~\ref{lemma:CM} implies that $S/(Q_1\cap Q_3)$ is Cohen-Macaulay. Next, we claim that
\[
(Q_1\cap Q_3)+Q_2\ =\ Q_3+Q_2.
\]
Assuming the claim, one has
\[
\dim S/(Q_1\cap Q_3)\ =\ \dim S/Q_2\ =\ \dim S/((Q_1\cap Q_3)+Q_2)+1,
\]
so Lemma~\ref{lemma:CM} shows that $S/(Q_1\cap Q_2\cap Q_3)$ is Cohen-Macaulay.

The verification of the claim reduces immediately to
\[
Q_3\ \subseteq\ (Q_1\cap Q_3)+Q_2,
\]
which, in turn reduces to
\[
I_m(Y|_{a+1})\ \subseteq\ (Q_1\cap Q_3)+Q_2.
\]
Since the ideals on the right contain $J_{a+1}$, it suffices to show that
\[
I_m(Y'|_{a+1})\ \subseteq\ (Q_1\cap Q_3)+Q_2.
\]
But
\[
I_m(Y'|_{a+1})\ =\ \frakP_{2m\times a+1}(Y'|_{a+1})+\frakQ_{2m\times a+1}(Y'|_{a+1}),
\]
and
\[
\frakP_{2m\times a+1}(Y'|_{a+1})\ \subseteq\ (Q_1\cap Q_3),
\quad\text{ while }\quad
\frakQ_{2m\times a+1}(Y'|_{a+1})\ \subseteq\ Q_2.
\]

This concludes the proof that $S/(I_\sigma+J_a)$ is Cohen-Macaulay for $\sigma$ standard, $a=s_k<n$, and $d$ odd. When $d$ is even, the proof that the rings $S/(I'_\sigma+J_a)$ and $S/(I''_\sigma+J_a)$ are Cohen-Macaulay resembles the proof in Case (i); one does not have to separately consider the cases where $k=m-1$. 
\end{proof}

We record the main consequences of Theorem~\ref{theorem:symmetric:prs}:

\begin{theorem}
\label{theorem:symmetric:odd:nullcone}
Let $Y$ be a~$(2t+1)\times n$ matrix of indeterminates over a field $K$ of characteristic other than two. Set~$S\colonequals K[Y]$ and
\[
\frakS\colonequals (Y^\tr Y)S+I_{t+1}(Y).
\]
Then $S/\frakS$ is a Cohen-Macaulay integral domain, with
\[
\dim S/\frakS=\begin{cases}
2nt -\displaystyle{\binom{n}{2}}& \text{if }\ n\le t+1,\\
nt+\displaystyle{\binom{t+1}{2}}& \text{if }\ n\ge t.
\end{cases}
\]
\end{theorem}

\begin{proof}
If $n\le t$, take $\sigma=(0,1,2,\dots,n-1,n)$ in Theorem~\ref{theorem:symmetric:odd:nullcone}, so $m=n$ and
\[
\dim S/\frakS=n((2t+1)+n-n-1)-(1+2+\dots+(n-1))\ =\ 2nt -\displaystyle{\binom{n}{2}}.
\]
If $n>t$, take $\sigma=(0,1,2,\dots,t-1,n)$, in which case $m=t$, and the theorem gives
\[
\dim S/\frakS=t((2t+1)+n-t-1)-(1+2+\dots+(t-1))\ =\ nt+\binom{t+1}{2}.\qedhere
\]
\end{proof}

In the case of a symmetric bilinear form of even rank, i.e., when the number of rows of~$Y$ is even, we have the following theorem; note that if $n\le t-1$, then $\frakP=\frakS=\frakQ$.

\begin{theorem}
\label{theorem:symmetric:even:nullcone}
Let $Y$ be a~$2t\times n$ matrix of indeterminates over a field $K$ of characteristic other than two. Set~$S\colonequals K[Y]$ and $\frakS\colonequals (Y^\tr Y) + I_{t+1}(Y)$, and let $\frakP$ and $\frakQ$ be as in Definition~\ref{definition:p:q}.

If $n\le t-1$ then $\frakP=\frakS=\frakQ$, and $S/\frakS$ is a Cohen-Macaulay integral domain with
\[
\dim S/\frakS\ =\ 2nt -\displaystyle{\binom{n+1}{2}}.
\]

If $n\ge t$, then $S/\frakP$, $S/\frakQ$, and $S/(\frakP+\frakQ)$ are Cohen-Macaulay integral domains with
\[
\dim S/\frakP\ =\ nt+\displaystyle{\binom{t}{2}}\ =\ \dim S/\frakQ,
\quad\text{ and }\quad
\dim S/(\frakP+\frakQ)\ =\ nt-n-1+\displaystyle{\binom{t+1}{2}}.
\]
\end{theorem}

\begin{proof}
If $n\le t-1$, take $\sigma=(0,1,2,\dots,n-1,n)$ in Theorem~\ref{theorem:symmetric:prs}, so $m=n$ and
\[
\dim S/\frakS=n(2t+n-n-1)-(1+2+\dots+(n-1))\ =\ 2nt -\displaystyle{\binom{n+1}{2}}.
\]
If $n\ge t$, take $\sigma=(0,1,2,\dots,t-1,n)$, in which case $m=t$, and
\[
\dim S/\frakP=t(2t+n-t-1)-(1+2+\dots+(t-1))\ =\ nt+\binom{t}{2}.
\]
The case of $S/\frakQ$ is similar. Next, note that
\[
\frakP+\frakQ = (Y^\tr Y) + I_t(Y),
\]
and that if $n\ge t$, taking $\sigma=(0,1,2,\dots,t-2,n)$ in Theorem~\ref{theorem:symmetric:prs} gives
\begin{multline*}
\dim S/(\frakP+\frakQ)\ =\ (t-1)(2t+n-(t-1)-1)-(1+2+\dots+(t-2))\\
=\ nt-n-1+\displaystyle{\binom{t+1}{2}},
\end{multline*}
which completes the proof.
\end{proof}

\subsection{The purity of the embedding}

Finally, we are in a position to settle the $\O_d(K)$ case of Theorem~\ref{theorem:main}:

\begin{theorem}
\label{theorem:symmetric:purity}
Let $K$ be a field of positive characteristic $p$. Fix positive integers $d$ and $n$, and consider the inclusion $\phi\colon K[Y^\tr Y]\to K[Y]$ where $Y$ is a size $d\times n$ matrix of indeterminates. Then $\phi$ is pure if and only if
\begin{enumerate}[\quad \rm(1)]
\item $d=1$, or
\item $d=2$ and $p$ is odd, or 
\item $p=2$ and $n\le (d+1)/2$, or
\item $p$ is odd and $n\le (d+2)/2$.
\end{enumerate}
\end{theorem}

\begin{proof}
As with the other matrix families, if~$\phi\colon K[Y^\tr Y]\to K[Y]$ is pure for fixed $(n,d)$, then purity holds as well for the inclusion of the $K$-algebras corresponding to~$(n',d)$ with~$n'\le n$. Set $S\colonequals K[Y]$ and $R\colonequals K[Y^\tr Y]$, and note that $\frakm_RS=(Y^\tr Y)S$.

When $d=1$, the ring $R$ coincides with the Veronese subring $S^{(2)}$, and is hence a pure subring of $S$.

Next, consider the case where $d=2$ and $p$ is odd. In proving the purity, one may enlarge~$K$ so as to assume that it is algebraically closed. The special orthogonal group~$\SO_2(K)$ is then isomorphic to the torus $K^\times$, so $\O_2(K)$ is the extension of $\ZZ/2$ by a torus, hence linearly reductive; see also~\cite[Remark~8.2]{JS}. It follows that purity holds in case (2).

When $n\le (d+1)/2$, Theorem~\ref{theorem:symmetric:ci} implies that the ideal $\frakm_RS$ is generated by a regular sequence of length~$\binom{n+1}{2}$. Since this is also the dimension of $R$, it follows that $\phi$ is pure.

If $p=2$, suppose first that $d$ is odd, say $d=2t+1$. We need to verify that~$\phi$ is not pure if $n=t+2$. This follows from Theorem~\ref{theorem:symmetric:nullcone:char2} since $S/(\rad\frakm_RS)$ is Cohen-Macaulay and
\begin{equation}
\label{equation:symmetric:d:odd}
\dim R - \height\frakm_RS\ =\ \binom{n+1}{2}-\left[dn-nt-\binom{t+1}{2}\right]\ =\ 1.
\end{equation}
Similarly, when $d=2t+2$, it suffices to verify that $\phi$ is not pure in the case $n=t+2$. Theorem~\ref{theorem:symmetric:nullcone:char2} implies that $S/\!\rad\frakm_RS$ is Cohen-Macaulay, and that
\[
\dim R - \height\frakm_RS\ =\ \binom{n+1}{2}-\left[dn-n(t+1)-\binom{t+1}{2}\right]\ =\ 1,
\]
which completes the case $p=2$; specifically, the argument above is valid in the case $d=2$, where one has $t=0$.

In the remaining cases, $p$ is an odd prime, and $d$ is at least $3$. When $d=2t+1$, we need to check that $\phi$ is not pure in the case $n=t+2$. This is much the same as~\eqref{equation:symmetric:d:odd}, with Theorem~\ref{theorem:symmetric:odd:nullcone} providing the needful.

Suppose $d=2t$ and $t\ge 2$. It suffices to verify that $\phi$ is pure in the case $n=t+1$, and that it is not pure in the case $n=t+2$. In either case, the ring $R$ is regular, with $\dim R=\binom{n+1}{2}$, so the critical local cohomology module is
\[
H^{\binom{n+1}{2}}_{\frakm_R}(S)\ =\ H^{\binom{n+1}{2}}_{\frakP\cap\frakQ}(S).
\]
By Theorem~\ref{theorem:symmetric:even:nullcone}, the ideals $\frakP$, $\frakQ$, and $\frakP+\frakQ$, define Cohen-Macaulay rings, and
\[
\height\frakP = nt-\binom{t}{2} = \height\frakQ
\quad\text{ and }\quad
\height(\frakP + \frakQ) = nt+n+1-\binom{t+1}{2}.
\]
When $n=t+1$, the Mayer-Vietoris sequence
\[
\CD
@>>> H^{\binom{n+1}{2}}_{\frakP\cap\frakQ}(S) @>>> H^{\binom{n+1}{2}+1}_{\frakP+\frakQ}(S) @>>> H^{\binom{n+1}{2}+1}_{\frakP}(S)\oplus H^{\binom{n+1}{2}+1}_{\frakQ}(S) @>>>
\endCD
\]
shows that $H^{\binom{n+1}{2}}_{\frakP\cap\frakQ}(S)$ is nonzero, since the middle term is nonzero and the term to the right vanishes. When $n=t+2$, the vanishing of $H^{\binom{n+1}{2}}_{\frakP\cap\frakQ}(S)$ follows from the vanishing of the outer terms in the exact sequence
\[
\CD
@>>> H^{\binom{n+1}{2}}_{\frakP}(S)\oplus H^{\binom{n+1}{2}}_{\frakQ}(S) @>>> H^{\binom{n+1}{2}}_{\frakP\cap\frakQ}(S) @>>> H^{\binom{n+1}{2}+1}_{\frakP+\frakQ}(S) @>>>.
\endCD\qedhere
\]
\end{proof}

In the case that the field $K$ has characteristic two, it is also reasonable to ask when the inclusion $K[Y^\tr Y,\ \sum_i y_{ij}\ |\ 1\le j\le n] \subseteq K[Y]$ is pure; we record the answer:

\begin{theorem}
\label{theorem:symmetric:purity:char:2}
Let $K$ be a field of characteristic two. Fix positive integers $d$ and $n$, and consider a $d\times n$ matrix of indeterminates $Y$. Then the inclusion
\[
K[Y^\tr Y,\ \sum_i y_{ij}\ |\ 1\le j\le n]\ \subseteq\ K[Y]
\]
is pure if and only if $d=1$ or $n\le (d+1)/2$.
\end{theorem}

\begin{proof}
The ring $R\colonequals K[Y^\tr Y,\ \sum_i y_{ij}\ |\ 1\le j\le n]$ is an integral extension of the symmetric determinantal ring $K[Y^\tr Y]$, and hence has the same dimension as $K[Y^\tr Y]$. Also, when~$K[Y^\tr Y]$ is regular, so is $R$. Set $S\colonequals K[Y]$. By Theorem~\ref{theorem:symmetric:nullcone:char2}, the ideal 
\[
\frakm_RS\ =\ (Y^\tr Y)S+(y_{11}+\dots+y_{d1},\ \dots,\ y_{1n}+\dots+y_{dn})S
\]
defines a Cohen-Macaulay ring $S/\frakm_RS$.

If $d=1$ then $R=S$. Assume $d\ge 2$, and express $d$ as $2t+1$ or $2t+2$, for $t$ an integer. Using the reduction as in the proof of Theorem~\ref{theorem:symmetric:purity}, it suffices to verify that $R\subseteq S$ is pure in the case $n=t+1$, and that it is not pure in the case $n=t+2$. In either case the ring $R$ is regular with $\dim R=\binom{n+1}{2}$, and the critical local cohomology module is $H^{\dim R}_{\frakm_RS}(S)$. Using Theorem~\ref{theorem:symmetric:nullcone:char2}, this module is nonzero in the case $n=t+1$ since $\height\frakm_RS=\binom{n+1}{2}$, whereas, if $n=t+2$, then
\[
\dim R - \height\frakm_RS\ =\ \binom{n+1}{2}-\left[n(t+1)-\binom{t+1}{2}\right]\ =\ 1,
\]
so $H^{\dim R}_{\frakm_RS}(S)=0$.
\end{proof}

\section*{Acknowledgments}

Several of the results were verified using the computer algebra systems \texttt{Macaulay2}~\cite{Macaulay2} and \texttt{Magma}~\cite{Magma}; the use of these is gratefully acknowledged. We are also indebted to the referee for a careful reading of the manuscript, and for helpful comments.



\begin{thebibliography}{HMSW}
\bibitem[Av]{Avramov}
L.~L.~Avramov, \emph{A class of factorial domains}, Serdica~\textbf{5} (1979), 378--379.

\bibitem[BCP]{Magma}
W.~Bosma, J.~Cannon, and C.~Playoust, \emph{The Magma algebra system. I. The user language}, J. Symbolic Comput.~\textbf{24} (1997), 235--265.

\bibitem[Bo]{Boutot}
J.-F.~Boutot, \emph{Singularit\'es rationnelles et quotients par les groupes r\'eductifs}, Invent. Math.~\textbf{88} (1987), 65--68.

\bibitem[Br1]{Bruns:cl}
W.~Bruns, \emph{Die Divisorenklassengruppe der Restklassenringe von Polynomringen nach Determinantenidealen}, Rev. Roumaine Math. Pures Appl.~\textbf{20} (1975), 1109--1111.

\bibitem[Br2]{Bruns:compositio}
W.~Bruns, \emph{Generic maps and modules}, Compositio Math.~\textbf{47} (1982), 171--193.

\bibitem[BH]{Bruns-Herzog}
W.~Bruns and J.~Herzog, \emph{Cohen-Macaulay rings}, revised edition, Cambridge Stud. Adv. Math.~\textbf{39}, Cambridge Univ. Press, Cambridge, 1998.

\bibitem[BE]{Buchsbaum-Eisenbud:1975}
D.~A.~Buchsbaum and D.~Eisenbud, \emph{Generic free resolutions and a family of generically perfect ideals}, Adv. Math.~\textbf{18} (1975), 245--301.

\bibitem[CGG]{CGG}
M.~V.~Catalisano, A.~V.~Geramita, and A.~Gimigliano, \emph{Secant varieties of Grassmann varieties}, Proc. Am. Math. Soc.~\textbf{133} (2005), 633--642.

\bibitem[CW]{CW}
A.~Conca and V.~Welker, \emph{Lov\'asz-Saks-Schrijver ideals and coordinate sections of determinantal varieties}, Algebra Number Theory~\textbf{13} (2019), 455--484.

\bibitem[DP]{DeConcini-Procesi}
C.~De~Concini and C.~Procesi, \emph{A characteristic free approach to invariant theory}, Adv. Math.~\textbf{21} (1976), 330--354. 

\bibitem[DS]{DeConcini-Strickland}
C.~De~Concini and E.~Strickland, \emph{On the variety of complexes}, Adv. Math.~\textbf{41} (1981), 57--77.

\bibitem[EN]{Eagon-Northcott}
J.~A.~Eagon and D.~G.~Northcott, \emph{Ideals defined by matrices and a certain complex associated with them}, Proc. Roy. Soc. London Ser. A~\textbf{269} (1962), 188--204.

\bibitem[Go1]{Goto1}
S.~Goto, \emph{The divisor class group of a certain Krull domain}, J. Math. Kyoto Univ.~\textbf{17} (1977), 47--50.

\bibitem[Go2]{Goto2}
S.~Goto, \emph{On the Gorensteinness of determinantal loci}, J. Math. Kyoto Univ.~\textbf{19} (1979), 371--374.

\bibitem[GS]{Macaulay2}
D.~R.~Grayson and M.~E.~Stillman, \emph{Macaulay2, a software system for research in algebraic geometry}, available at \url{http://www.math.uiuc.edu/Macaulay2/}.

\bibitem[Ha]{Hashimoto}
M.~Hashimoto, \emph{Another proof of theorems of De~Concini and Procesi}, J. Math. Kyoto Univ.~\textbf{45} (2005), 701--710.

\bibitem[HMSW]{HMSW}
J.~Herzog, A.~Macchia, S.~Saeedi Madani, and V.~Welker, \emph{On the ideal of the orthogonal representations of a graph in $\mathbb{R}^2$}, Adv.~in~Appl.~Math.~\textbf{71} (2015), 146--173.

\bibitem[He]{Hesselink}
W.~Hesselink, \emph{Desingularizations of varieties of nullforms}, Invent. Math.~\textbf{55} (1979), 141--163.

\bibitem[Hi]{Hilbert}
D.~Hilbert, \emph{\"Uber die vollen Invariantensysteme}, Math. Ann.~\textbf{42} (1893), 313--373.

\bibitem[Ho1]{Hochster:schubert}
M. Hochster, \emph{Grassmannians and their Schubert subvarieties are arithmetically Cohen-Macaulay}, J. Algebra~\textbf{25} (1973), 40--57.

\bibitem[Ho2]{Hochster:solid1}
M. Hochster, \emph{Solid closure}, Contemp. Math. \textbf{159} (1994), 103--172.

\bibitem[HE]{Hochster-Eagon}
M.~Hochster and J.~A.~Eagon, \emph{Cohen-Macaulay rings, invariant theory, and the generic perfection of determinantal loci}, Amer. J. Math. \textbf{93} (1971), 1020--1058.

\bibitem[HH1]{HH:JAMS}
M.~Hochster and C.~Huneke, \emph{Tight closure, invariant theory, and the Brian\c con-Skoda theorem}, J. Amer. Math. Soc.~\textbf{3} (1990), 31--116.

\bibitem[HH2]{HH:JAG}
M.~Hochster and C.~Huneke, \emph{Tight closure of parameter ideals and splitting in module-finite extensions}, J. Algebraic Geom.~\textbf{3} (1994), 599--670.

\bibitem[HR]{Hochster-Roberts}
M.~Hochster and J.~L.~Roberts, \emph{Rings of invariants of reductive groups acting on regular rings are Cohen-Macaulay}, Adv. Math.~\textbf{13} (1974), 115--175.

\bibitem[HP]{HP} W.~V.~D.~Hodge and D.~Pedoe, \emph{Methods of Algebraic Geometry. Vol. I}, Cambridge, at the University Press; New York, 1947.

\bibitem[Hu]{Huneke:TAMS}
C.~Huneke, \emph{The arithmetic perfection of Buchsbaum-Eisenbud varieties and generic modules of projective dimension two}, Trans. Amer. Math. Soc.~\textbf{265} (1981), 211--233.

\bibitem[Ig]{Igusa}
J.-i. Igusa, \emph{On the arithmetic normality of the Grassmann variety}, Proc. Nat. Acad. Sci. U.S.A.~\textbf{40} (1954), 309--313.

\bibitem[Ja]{Jagy}
W.~Jagy, (user \href{https://mathoverflow.net/users/3324/will-jagy}{3324}), \emph{What's your favorite equation, formula, identity or inequality?} \href{https://mathoverflow.net/q/17139}{Mathoverflow}, 2014.

\bibitem[JS]{JS}
J.~Jeffries and A.~K.~Singh, \emph{Differential operators on classical invariant rings do not lift modulo~$p$},\newline \url{https://arxiv.org/abs/2006.03029}.

\bibitem[Ke1]{Kempf:BAMS}
G.~R.~Kempf, \emph{Images of homogeneous vector bundles and varieties of complexes}, Bull. Amer. Math. Soc.~\textbf{81} (1975), 900--901.

\bibitem[Ke2]{Kempf:Invent}
G.~R.~Kempf, \emph{On the collapsing of homogeneous bundles}, Invent. Math.~\textbf{37} (1976), 229--239.

\bibitem[Ke3]{Kempf:MMJ}
G.~R.~Kempf, \emph{The Hochster-Roberts theorem of invariant theory}, Michigan Math. J.~\textbf{26} (1979), 19--32.

\bibitem[KL]{Kleppe-Laksov}
H.~Kleppe and D.~Laksov, \emph{The algebraic structure and deformation of Pfaffian schemes}, J. Algebra~\textbf{64} (1980), 167--189.

\bibitem[Ko]{Kohls}
M.~Kohls, \emph{Non-Cohen--Macaulay invariant rings of infinite groups}, J. Algebra~\textbf{306} (2006), 591--609.

\bibitem[KS]{Kraft-Schwarz}
H.~Kraft and G.~W.~Schwarz, \emph{Representations with a reduced null cone}, in: Symmetry: representation theory and its applications, Progr. Math.~\textbf{257}, pp.~419--474, Birkh\"auser/Springer, New York, 2014.

\bibitem[KW]{Kraft-Wallach}
H.~Kraft and N.~R.~Wallach, \emph{On the nullcone of representations of reductive groups}, Pacific J. Math.~\textbf{224} (2006), 119--139. 

\bibitem[Ku]{Kutz}
R.~Kutz, \emph{Cohen-Macaulay rings and ideal theory in rings of invariants of algebraic groups}, Trans. Amer. Math. Soc.~\textbf{194} (1974), 115--129.

\bibitem[La]{Laksov}
D.~Laksov, \emph{The arithmetic Cohen-Macaulay character of Schubert schemes}, Acta Math.~\textbf{129} (1972), 1--9. 

\bibitem[Ma1]{Marinov1}
V.~P.~Marinov, \emph{Perfection of ideals generated by the Pfaffians of an alternating matrix, I}, Serdica. Bulgaricae Mathematicae Publicationes~\textbf{9} (1983), 31--42.

\bibitem[Ma2]{Marinov2}
V.~P.~Marinov, \emph{Perfection of ideals generated by the Pfaffians of an alternating matrix, II}, Serdica. Bulgaricae Mathematicae Publicationes~\textbf{9} (1983), 122--131.

\bibitem[Mu]{Musili}
C.~Musili, \emph{Postulation formula for Schubert varieties}, J. Indian Math. Soc. (N.S.)~\textbf{36} (1972), 143--171.

\bibitem[MS]{Musili:Seshadri}
C.~Musili and C.~S.~Seshadri, \emph{Schubert varieties and the variety of complexes}, in Arithmetic and geometry Vol. II, Progr. Math.~\textbf{36}, pp.~329--359, Birkh\"auser Boston, Boston, MA, 1983.

\bibitem[PS]{PS}
C.~Peskine and L.~Szpiro, \emph{Dimension projective finie et cohomologie locale}, Inst. Hautes \'Etudes Sci. Publ. Math.~\textbf{42} (1973), 47--119.

\bibitem[Pr]{Procesi}
C.~Procesi, \emph{Lie groups, An approach through invariants and representations}, Universitext, Springer, New York, 2007.

\bibitem[Ri]{Richman}
D.~R.~Richman, \emph{The fundamental theorems of vector invariants}, Adv. Math.~\textbf{73} (1989), 43--78.

\bibitem[Ro]{Roberts:lc}
P.~C.~Roberts, \emph{A computation of local cohomology}, Contemp. Math. \textbf{159} (1994), 351--356.

\bibitem[Sc]{Schwarz}
G.~W.~Schwarz, \emph{Representations of simple Lie groups with a free module of covariants}, Invent. Math.~\textbf{50} (1978/79), 1--12.

\bibitem[Sv]{Svanes}
T.~Svanes, \emph{Coherent cohomology on Schubert subschemes of flag schemes and applications}, Adv. Math.~\textbf{14} (1974), 369--453.

\bibitem[Tc]{Tchernev}
A.~B.~Tchernev, \emph{Universal complexes and the generic structure of free resolutions}, Michigan Math. J.~\textbf{49} (2001), 65--96.

\bibitem[We]{Weyl}
H.~Weyl, \emph{The classical groups. Their invariants and representations}, Princeton University Press, Princeton, NJ, 1997.

\end{thebibliography}
\end{document}